\documentclass[a4paper]{amsart}

\usepackage{amsfonts, amssymb, amsmath, amsthm}                               
\usepackage{graphicx}
\usepackage[font=scriptsize,labelfont=bf,margin=3cm]{caption}
\usepackage{epstopdf}
\usepackage{pdfpages}
\usepackage{pdflscape}
\usepackage{bbm}
\usepackage{mathtools}
\usepackage{url}                                                              
\usepackage{color}
\usepackage{slashed}
\usepackage{comment}
\usepackage{enumerate}
\usepackage[font=scriptsize,labelfont=bf,margin=3cm]{caption}

\newtheorem{theorem}{Theorem}[section]      	      	                        
\newtheorem{corollary}[theorem]{Corollary}     	      	      	      	      
\newtheorem{lemma}[theorem]{Lemma}     	       	      	      	      	      
\newtheorem{proposition}[theorem]{Proposition} 	      	      	      	      
\newtheorem{definition}[theorem]{Definition}   	      	                      
\newtheorem*{problem}{Problem}   	      	                                    
\newtheorem*{setting}{Setting}   	      	                                    
\theoremstyle{remark}
\newtheorem{remark}[theorem]{Remark}

\numberwithin{equation}{section}                                              
\numberwithin{theorem}{section}                                               
\numberwithin{figure}{section}                                                

\newcommand{\mf}[1]{\mathfrak{#1}}                                            
\newcommand{\mc}[1]{\mathcal{#1}}                                             

\newcommand{\N}{\mathbb{N}}                                                   
\newcommand{\R}{\mathbb{R}}                                                   
\newcommand{\Sph}{\mathbb{S}}                                                 

\newcommand{\snabla}{\slashed{\nabla}}                                        
\newcommand{\sDe}{\slashed{\Delta}}                                           
\newcommand{\pd}{\partial}                                                    
\newcommand{\nab}{\nabla}                                                     

\newcommand{\be}{\beta}
\newcommand{\de}{\delta}
\newcommand{\ep}{\varepsilon}

\newcommand{\ka}{\kappa}
\newcommand{\la}{\lambda}

\newcommand{\si}{\sigma}
\newcommand{\te}{\theta}
\newcommand{\vp}{\varphi}

\newcommand{\De}{\Delta}
\newcommand{\Ga}{\Gamma}

\newcommand{\Om}{\Omega}

\renewcommand\leq\leqslant
\renewcommand\geq\geqslant

\DeclareMathOperator\Div{\nabla\cdot}
\renewcommand\pi\Pi

\allowdisplaybreaks

\begin{document} 

\title[Control of parabolic equations with inverse square potentials]{Controllability of parabolic equations\\ with inverse square infinite potential wells\\ via global Carleman estimates}
	
\author{Alberto Enciso}
\address{Instituto de Ciencias Matem\'aticas, Consejo Superior de Investigaciones Cient\'\i ficas, Madrid, Spain}
\email{aenciso@icmat.es}

\author{Arick Shao}
\address{School of Mathematical Sciences, Queen Mary University, London, United Kingdom}
\email{a.shao@qmul.ac.uk}
	
\author{Bruno Vergara}
\address{Department of Mathematics, Brown University, Kassar House, 151 Thayer St., Providence, RI 02912, USA}
\email{bruno\textunderscore vergara\textunderscore biggio@brown.edu}

\begin{abstract}
We consider heat operators on a convex domain $\Om$, with a critically singular potential that diverges as the inverse square of the distance to the boundary of $\Om$.
We establish a general boundary controllability result for such operators in all dimensions, in particular providing the first such result in more than one spatial dimension.
The key step in the proof is a new global Carleman estimate with a carefully chosen weight that captures the appropriate boundary conditions, the global geometry of the domain $\Omega$, and the $H^1$-energy for this problem.
The estimate is derived by combining two intermediate Carleman inequalities with distinct and carefully constructed weights involving non-smooth powers of the boundary distance.
\end{abstract}

\maketitle

\section{Introduction} \label{S.Intro}
	
In this paper, we consider, on a bounded domain in $\R^n$, the heat operator with a potential that diverges as the inverse square of the distance to the boundary hypersurface.
More precisely, our setting will be the following:

\begin{setting} \label{ass.domain}
Throughout the paper, we let $\Ga$ denote a closed, connected, and convex $C^4$-hypersurface in $\R^n$ ($n \geq 1$), and we let $\Omega$ denote the interior domain that is bounded by $\Ga$.
In addition, we let $d_\Ga: \Omega \rightarrow \R$ denote the distance to $\Ga$.
\end{setting}

We will consider the following equation on $\Omega$ and over a time interval:
\begin{equation}
\label{eq.0} -\pd_t v + \Big( \De + \frac{ \sigma }{ d_\Ga^2 } \Big) v + Y \cdot \nab v + W \, v = 0 \text{.}
\end{equation}
Here, $\si \in \R$ is a parameter measuring the strength of the singular potential, while $Y$ and $W$ represent first and zero-order coefficients that are less singular at $\Ga$.

Our main objective in this paper is to derive boundary null controllability for the above equation.
Given any initial state $v (0)$ and $T>0$, the question is whether one can pick some control $f$ on the boundary $( 0, T ) \times \Ga$ so that the evolution through \eqref{eq.0}---together with the boundary control---drives the solution to the target state $v(T)=0$ at time $T$.
While results have been established in one spatial dimension using moment methods (see \cite{Biccari} and references therein), here we provide, to our best knowledge, the first such result for general domains in arbitrary dimensions.

To show the above property, we derive sharp Carleman estimates for the operator in \eqref{eq.0}.
Indeed, genuinely new estimates are needed, since the singular potential scales as the Laplacian near the hypersurface $\Ga$, hence one cannot treat \eqref{eq.0} as a perturbation of the standard heat equation.
Moreover, these estimates will enable us to obtain robust boundary observability and controllability results, in that we both treat any spatial dimension and deal with a large class of lower-order coefficients.
Note the inclusion of $Y$ and $W$ in \eqref{eq.0} is very natural in our context, as $d_\Ga$ itself fails to be regular and well-behaved away from $\Ga$.
	
\subsection{Boundary asymptotics}

Let us start by describing the role of the strength parameter $\si$ in the boundary asymptotics of solutions to \eqref{eq.0}.
We let $\si<\frac14$, and we consider the Cauchy problem associated to \eqref{eq.0}, with initial data $v (0) = v_0$.
Moreover, it will often be convenient to write $\sigma := \ka ( 1 - \ka )$, with $\ka := \ka (\si) < \tfrac{1}{2}$.
        
According to the classical Frobenius theory for ODEs, the inverse-square singularity of the potential at $d_\Ga=0$ implies the characteristic exponents of this equation are precisely $\ka$ and $1-\ka$.
Therefore, if $\ka$ is not a half-integer (which ensures that logarithmic branches will not appear), solutions to the equation are expected to behave either like $d_\Ga^\ka$ or $d_\Ga^{1-\ka}$ close to $\Ga$ and correspond to the Dirichlet and Neumann branches, respectively.
As a result, the boundary data for our problem must be formulated with these $d_\Ga$-weights taken into account.

Now, as such quantities will naturally appear throughout the article, we set the following notations for future convenience:
 
\begin{definition} \label{kappa_def}
Given a strength parameter $\sigma \in ( -\infty, \frac{1}{4} )$:
\begin{itemize}
\item We let $\kappa := \kappa (\si) \in \R$ be the unique parameter satisfying
\begin{equation}
\label{kappa} \sigma := \ka (1 - \ka ) \text{,} \qquad \kappa < \tfrac{1}{2} \text{.}
\end{equation}

\item We define the associated Dirichlet and Neumann trace operators:
\begin{equation}
\label{bdry_data} \mc{D}_\sigma \phi := d_\Ga^{-\ka} \phi |_{ d_\Ga \searrow 0 } \text{,} \qquad \mc{N}_\sigma \phi := d_\Ga^{2\ka} \nab d_\Ga \cdot \nab( d_\Ga^{-\kappa} \phi ) |_{ d_\Ga \searrow 0 } \text{.}
\end{equation}

\item In addition, we introduce the following notation:
\begin{equation}\label{operator}
\De_\si := \De + \frac{\si}{ d_\Ga^2} \text{.}
\end{equation}
\end{itemize}
\end{definition}

\begin{remark}
We stress that throughout the paper, $\ka$ will always implicitly depend on $\si$ via the relation \eqref{kappa}.
Note that there is a one-to-one correspondence between the values of $\sigma \in ( -\infty, \frac{1}{4} )$ and $\kappa \in ( -\infty, \frac{1}{2} )$.
In particular:
\begin{align*}
\sigma \nearrow \tfrac{1}{4} \qquad &\leftrightarrow \qquad \kappa \nearrow \tfrac{1}{2} \text{,} \\
\sigma = 0 \qquad &\leftrightarrow \qquad \kappa = 0 \text{,} \\
\sigma = -\tfrac{3}{4} \qquad &\leftrightarrow \qquad \kappa = -\tfrac{1}{2} \text{.}
\end{align*}
In addition, all the associated quantities in \eqref{eq.0} and Definition \ref{kappa_def} reduce to the standard ones in the absence of the singular potential, i.e., when $\si=0$.
\end{remark}

Later in this paper, we will show that the Dirichlet and Neumann traces in \eqref{bdry_data} indeed lead to viable well-posedness theories for \eqref{eq.0}, at least for a subset of values~$\sigma$; see Sections \ref{S.obs} and \ref{S.control}.
As a result, \eqref{bdry_data} provides natural notions of boundary data for our upcoming main boundary control results.

The specific range of $\si$ for which we will develop well-posedness results is discussed further below.
For the moment, we note $\sigma = \frac{1}{4}$ can be viewed as a critical threshold, as \eqref{eq.0} is expected to be ill-posed for $\sigma > \frac{1}{4}$; see \cite{Baras, Biccari, VaZuazua}.
(Moreover, \cite{Biccari} showed---in one spatial dimension---that the cost of boundary control blows up in the limit $\sigma \nearrow \frac{1}{4}$.)
We also highlight $\sigma = -\frac{3}{4}$ as another natural threshold, since the Dirichlet branch fails to lie in $L^2$ once $\sigma \leq -\frac{3}{4}$.

\begin{remark}
Analogues of the adapted boundary data \eqref{bdry_data} have been considered before in the literature in different contexts for other singular operators; see, e.g., \cite{MM, Warnick}.
The boundary conditions \eqref{bdry_data} were also used in \cite{JEMS} toward Carleman and observability estimates for the wave equation analogue of \eqref{eq.0}.
\end{remark}

\subsection{Motivation}
	
Parabolic problems involving inverse square potentials have been intensively studied in the past decades; see \cite{Baras, Ireneo}, for instance, as well as references within for some early results.
Since the literature in this area is far too extensive to describe in full, we restrict our focus here to null controllability and Carleman estimates to keep the present discussion concise.

First, in one spatial dimension, in which we can set $\Omega := ( 0, 1 )$ without loss of generality, there are ample results treating the singular heat operator
\begin{equation}
\label{eq.0_1d} - \partial_t + \partial_x^2 + \tfrac{ \sigma }{ x^2 } \text{.}
\end{equation}
For instance, interior null controllability results for \eqref{eq.0_1d}---with the control supported away from $x = 0$---were established in \cite{CMV, CMV2, CMV5, MV}.
Also, various boundary null controllability results for \eqref{eq.0_1d} have been proven, both at $x = 1$ (away from the singularity) \cite{CMV5} and at $x = 0$ (at the singularity) \cite{Biccari, CMV4, Gueye}.

\begin{remark}
Many of the above results treated the degenerate parabolic operator
\begin{equation} \label{degen}
- \partial_t + \partial_x ( x^\alpha \partial_x \, \cdot \, ) \text{,} \qquad \alpha \in ( 0, 2 ) \text{.}
\end{equation}
However, this can be transformed to \eqref{eq.0_1d} through an appropriate change of variables, at least for a subset of parameters $\sigma$; see \cite[Appendix A]{Biccari} for details.
\end{remark}

Of particular relevance is the recent result of Biccari \cite{Biccari}, which established boundary null controllability at $x = 0$ for \eqref{eq.0_1d}, with $\sigma < \frac{1}{4}$.
As \cite{Biccari} applied the moment method, which relied strongly upon an eigenfunction decomposition of $\partial_x^2 + \sigma x^{-2}$, the results do not readily extend to higher dimensions, nor to parabolic equations with general lower-order terms as in \eqref{eq.0}.
Partly for this reason, the author listed several open questions of interest; see \cite[Section 8]{Biccari}.

A key motivation of the present work is to address a number of these points:
\begin{enumerate}
\item We use \emph{Carleman estimates} to prove our controllability result.
Such techniques have the advantage of being more robust, in that they allow one to treat lower-order terms and to more easily extend to nonlinear problems.

\item We treat the case where the potential diverges on all of $\Ga$.
As mentioned in \cite{Biccari}, even in one spatial dimension, the case of a potential singular at both $x = 0$ and $x = 1$ cannot be treated via the moment method.

\item We obtain boundary null controllability in \emph{all spatial dimensions}, under the assumption $\Ga$ is convex.
To our knowledge, this is the first such boundary control result in higher dimensions; see the discussions below.
\end{enumerate}
In particular, \cite{Biccari} highlighted the problem of developing Carleman estimates adapted to the weighted boundary data \eqref{bdry_data} as being especially challenging.

Next, turning to higher dimensions (with general $\Omega \subseteq \R^n$), \cite{CMV3, Ervedoza, VZuazua} established interior controllability results for the singular heat operator
\begin{equation}
\label{eq.0_pt} - \partial_t + \Delta + \tfrac{ \sigma }{ | x - x_0 |^2 } \text{,} \qquad x_0 \in \Omega \text{,}
\end{equation}
i.e., a singular potential that diverges as an inverse square of the distance to a single interior point.
The above results were then extended in \cite{Cazacu} to the case $x_0 \in \Ga$, in which the potential instead diverges at a single boundary point.

The case of higher dimensional settings \eqref{eq.0}, where the potential becomes singular on all of $\Ga$, is known to be particularly difficult.
Incidentally, these arise naturally when considering parabolic equations on conformally compact Riemannian manifolds; see, e.g., \cite{Vazquez}.
Along this direction, Biccari and Zuazua \cite{BZuazua} first proved interior null controllability for the operator $-\partial_t + \Delta_\sigma$ using Carleman estimates.

The authors in \cite{BZuazua} stress that one cannot employ their results to derive boundary controllability or boundary observability properties.
The key reason is that their Carleman estimates do not capture an appropriate notion of the Neumann data at the boundary, \eqref{bdry_data} in particular.
Moreover, the Carleman estimate in \cite{BZuazua} only captures the full $L^2$-norm, and not the (unweighted) $H^1$-norm; as we shall see, the full $H^1$-norm will be a critical part of our setup.

\subsection{Boundary controllability and observability}

In this subsection, we state the main results of this paper.
However, before doing so, we first give a precise description of the lower-order coefficients $Y$ and $W$ in \eqref{eq.0}:

\begin{definition} \label{lower_order_def}
We let $\mc{Z}$ denote the collection of all pairs $( Y, W )$, where:
\begin{itemize}
\item $Y: \Omega \rightarrow \R^n$ is a $C^1$-vector field, and $W: \Omega \rightarrow \R$ is an $L^\infty$-function.

\item $Y$ extends to a $C^3$-vector field on a neighborhood of $\Ga$.

\item $d_\Ga W$ extends to a $C^2$-function on a neighborhood of $\Ga$.
\end{itemize}
\end{definition}

While the exact form of Definition \ref{lower_order_def} is technical in nature, at an informal level, our results will require $Y$ and $W$ to have sufficient regularity at $\Ga$.
On other other hand, since $d_\Ga$ fails to be regular away from $\Ga$, then it will also be useful for $Y$ and $W$ to be less regular away from $\Ga$; see Remark \ref{wp_y}.
Though the conditions in Definition \ref{lower_order_def} are not optimal, we adopt these particular assumptions in their current form, since they allow for a simpler presentation.

The main result of this paper is the boundary null controllability for the singular parabolic equation \eqref{eq.0}.
More precisely, we consider the following Cauchy problem:

\begin{problem}[C]
Given initial data $v_0$ on $\Omega$, as well as Dirichlet boundary data $f$ on $( 0, T ) \times \Ga$, solve the initial-boundary value problem for $v$,
\begin{align}
\label{heat_ctl}  -\partial_t v + \Delta_\sigma v + Y \cdot \nabla v + W v = 0 &\quad \text{on $( 0, T ) \times \Omega$,} \\
\notag  v ( 0 ) = v_0 &\quad \text{on $\Omega$,} \\
\notag \mc{D}_\sigma v = f &\quad \text{on $( 0, T ) \times \Gamma$} \text{,}
\end{align}
where $\sigma \in ( -\frac{3}{4}, 0 )$, and where the lower-order coefficients satisfy $( Y, W ) \in \mc{Z}$.
\end{problem}
	
The following statement, which is a simplification of the more precise Theorem \ref{T.control} in the main text, represents our main boundary control result:

\begin{theorem} \label{T.control0}
Let $\Omega \subseteq \R^n$ be a bounded domain, with a convex, connected, $C^4$-boundary $\Ga$, and fix $\sigma \in ( -\frac{3}{4}, 0 )$.
Then, Problem (C) is boundary null controllable in any positive time---for any initial data $v_0 \in L^2(\Om)$ and any $T > 0$, there exists Dirichlet boundary data $f \in L^2((0,T) \times \Ga)$ such that the corresponding solution $v$ to \eqref{heat_ctl}, with the above $v_0$ and $f$, satisfies $v(T) \equiv 0$.
\end{theorem}

Theorem \ref{T.control0} is, to our best knowledge, the first boundary controllability result for $(-\pd_t + \De_\si ) v = 0$ in spatial dimensions higher than $1$, and for the equation \eqref{eq.0}---containing also general lower-order terms---in any dimension.

To prove Theorem \ref{T.control0}, we employ (the variational formulation of) the celebrated Hilbert uniqueness method (HUM); see \cite{Lions1, MZuazua}.
As is standard, the main step is to obtain key estimates---most crucially an appropriate \emph{observability inequality}---for the dual problem.
Thus, in the context of observability, we will consider the following Cauchy problem for the backwards singular heat equation:

\begin{problem}[O]
Given final data $u_T$ on $\Omega$, solve the following problem for $u$,
\begin{align}
\label{heat_obs}  \partial_t u + \Delta_\sigma u + X \cdot \nabla u + V u = 0 &\quad \text{on $( 0, T ) \times \Omega$,} \\
\notag u ( T ) = u_T &\quad \text{on $\Omega$,} \\
\notag u = 0 &\quad \text{on $( 0, T ) \times \Gamma$} \text{,}
\end{align}
where $\sigma \in (-\frac{3}{4}, 0 )$, and where the lower-order coefficients satisfy $( X, V ) \in \mc{Z}$.
\end{problem}

\begin{remark}
Since $\sigma < 0$, the boundary condition in \eqref{heat_obs} implies $\mc{D}_\sigma u = 0$.
While one could develop an equivalent theory using instead the condition $\mc{D}_\sigma u = 0$, here we remain with $u = 0$ to be consistent with the existing literature, e.g., \cite{Biccari, BZuazua}.
\end{remark}

A first consideration in the proof of Theorem \ref{T.control0} is finding an appropriate choice of spaces for the controllability problem (C), as well as the corresponding spaces for the observability problem (O).
To apply the HUM, one requires estimates for an appropriately defined Neumann trace in Problem (O):
\begin{itemize}
\item The Neumann trace should be bounded by the final data $u_T$.

\item The Neumann trace should satisfy a boundary observability estimate, that is, it should be bounded from below by $u (0)$.
\end{itemize}
We will show that the above indeed holds when $u_T$ lies in the usual energy space.

In particular, in Section \ref{S.obs}, we briefly summarize the well-posedness theory for Problem (O) with final data $u_T \in H^1_0 ( \Omega )$---that is, the analogue of strict solutions in \cite{Biccari, BZuazua}.
We then show (in Proposition \ref{neumann_trace}) that if $\sigma \in ( -\frac{3}{4}, \frac{1}{4} )$, then the quantity $\mc{N}_\sigma u$ from \eqref{bdry_data} is indeed well-defined and bounded in $L^2$ by the $H^1$-norm of $u_T$.
Furthermore, if $\sigma < 0$ as well, then we prove (in Theorem \ref{obs}) observability by bounding $\mc{N}_\sigma u$ in $L^2$ from below by the $H^1$-norm of $u (0)$.
The above two estimates can be roughly summarized by the following theorem:

\begin{theorem} \label{T.Observability0}
Let $\Omega \subseteq \R^n$ be a bounded domain, with a convex, connected, $C^4$-boundary $\Ga$, and fix $\sigma \in ( -\frac{3}{4}, 0 )$.
Moreover, let $u$ be a solution to Problem (O), with final data $u_T \in H^1_0 (\Om)$.
Then, the Neumann data $\mc{N}_\sigma u$ on $\Ga$ is finite, and
\[
\int_\Om | \nab u(0) |^2 \lesssim \int_{ (0,T) \times \Ga } ( \mc{N}_\si u )^2 \lesssim \int_\Om | \nab u_T |^2 \text{.}
\]
\end{theorem}

Using the estimates of Theorem \ref{T.Observability0}, the boundary null controllability of Theorem \ref{T.control0} then follows by adapting standard duality arguments.
In Section \ref{S.control}, we develop the dual theory of weak (or transposition) ($H^{-1}$-)solutions for Problem (C), now in the presence of the singular potential.
We then show (in Theorem \ref{T.control}) that in this setting, one can construct the desired null boundary controls in $L^2$.

We stress that the well-posedness theories for Problems (C) and (O) are far from direct due to the singular potential, and are further complicated by the lower-order terms.
Thus, for completeness, we develop both theories in Sections \ref{S.obs} and \ref{S.control}.

The key ingredient to establishing the crucial Theorem \ref{T.Observability0} is a new global Carleman estimate, proved in this paper, that captures both the boundary data \eqref{bdry_data} and the $H^1$-energy of the solutions.
This is discussed in the following subsection.

\subsection{Global Carleman inequality}

Carleman estimates have found many applications in PDEs, such as in unique continuation \cite{Arons, Calderon, Carleman, DosSantos, EF, Hormander, Tataru1}, control theory \cite{DZZ, FI, LL, LLZ, Tataru_Thesis, Tataru0}, inverse problems \cite{BK0, BK, Klibanov}, and embedded eigenvalues in the continuous spectrum of Schr\"{o}dinger operators.

We next motivate and state the new global Carleman estimate for the singular parabolic operator $\pm \pd_t + \De_\si$.
The premier issue is that of capturing the Neumann boundary data from \eqref{bdry_data}, which now involves powers of $d_\Ga$ that blow up at $\Gamma$.
This is achieved through a specially constructed Carleman weight that is designed to generate precisely the correct power of $d_\Ga$ at $\Ga$.

In the Carleman estimate of \cite{BZuazua} (which yielded interior observability for solutions of $(\pd_t+\De_\si)u=0$), the authors employed a weight having, near $\Ga$, the form
\begin{equation}\label{CWZ}
f_0 (t, x) := \frac{1}{ t^3 (T-t)^3 } \left[ C - d_\Ga (x)^2 \, \psi (x) - \bigg( \frac{d_\Ga(x)}{d_0} \bigg)^s \, e^{s \, \psi(x)} \right] \text{,}
\end{equation}
with $s$ a large enough real number, $C$ and $d_0$ constants, and $\psi \in C^4 (\bar{\Om})$ a function of $d_\Ga$ so that $f_0$ is sufficiently smooth near and at $\Ga$.
For our case, we must replace~$s$ by a smaller power depending on $\ka$, so that appropriate singular weights appear upon differention.
In particular, near $\Ga$, our weight will be of the form
\begin{equation}
\label{weight_pre} F_0 ( t, x ) := \frac{1}{ t (T-t) } \left[ \frac{1}{ 1 + 2 \ka } \, d_\Ga (x)^{1+2\ka} + \beta \right] \text{,}
\end{equation}
with $\beta > 0$ a suitably chosen constant.
A key step will be then
to show that the weight \eqref{weight_pre} indeed suffices to capture the
Neumann trace from \eqref{bdry_data} on $\Ga$.

Next, observe that in order to prove Theorem \ref{T.Observability0}, our Carleman estimate will also need to control the $H^1$-norm on $\Om$.
In \cite{BZuazua}, their choice of weight \eqref{CWZ} yields control of a bulk quantity that is roughly of the form
\[
\int_0^T \frac{1}{ t^3 (T-t)^3 } \int_\Om d_\Ga^s | \nabla u |^2 \,dx\, dt \text{,} \qquad s > 0 \text{.}
\]
Because of the factor $d_\Ga(x)^s$, which vanishes near $\Ga$, their estimate fails to capture the full $H^1$-energy of $u$.
(Only the full $L^2$-norm was needed in \cite{BZuazua}.)

For our setting, we show that by using the weight \eqref{weight_pre} instead, we can capture the full $\dot{H}^1$-norm, without a weight that degenerates at $\Ga$.
Here, we note that our assumption of $\Ga$ being convex is crucial, as this ensures that the bulk terms in our Carleman estimate containing $| \nabla u |^2$ are uniformly positive.

Unfortunately, \eqref{weight_pre} does not yet suffice for a global Carleman estimate on all of $\Omega$.
This is because while $d_\Ga$ is $C^4$ near $\Ga$ (by the regularity of $\Ga$), it can fail to be differentiable elsewhere in $\Om$ due to the presence of caustics.
To get around this, we replace $d_\Ga$ by a more general \emph{boundary defining function} $y \in C^4 (\Om)$ that coincides with $d_\Ga$ in a thin neighborhood of $\Ga$.
While this function $y$ remains regular away from $\Ga$, it also
retains almost the same convexity properties as $d_\Ga$.
See Definition \ref{ep-bdf} for the precise properties of $y$, and Lemma \ref{L.convex} for its construction.

Thus, for our global Carleman weight, we replace $d_\Ga$ in \eqref{weight_pre} by $y$:
\begin{equation} \label{weight}
F (t,x) := \frac{1}{ t (T-t) } \left[ \frac{1}{ 1 + 2 \ka } \, y (x)^{1+2\ka} + \beta \right] \text{.}
\end{equation}
Since $y$ is ``close enough" to $d_\Ga$ for our purposes, by using $F$,
we both capture the Neumann trace and bound the global $\dot{H}^1$-norm on all of $\Om$ as desired.

We emphasize that the above still leaves untreated one fundamental issue---while these arguments suffice to control the $L^2$-norm of $\nabla u$, the same cannot be said for the $L^2$-norm of $u$ itself.
Away from the (unique by construction) critical point $x_\ast$ of $y$, our Carleman inequality allows us to control bulk integrals (over $( 0, T ) \times \Om$) of $u^2$ with uniformly positive weights, provided $\sigma \in (-\frac{3}{4}, 0)$.
However, these weights, which are accompanied by factors of $| \nabla y |^2$, can become non-positive near $x_\ast$.

To overcome this rather serious obstacle, we construct \emph{two} boundary defining functions $y_1$, $y_2$ with distinct critical points $x_{1,\ast} \neq x_{2,\ast}$; see Lemma \ref{L.convex}.
We then \emph{sum the two Carleman estimates} arising from $y_1$ and $y_2$.
In particular, the above-mentioned non-positivity for the $y_1$-Carleman estimate near $x_{1,\ast}$ can be overcome by a positive $L^2$-contribution in the $y_2$-estimate (since $x_{1,\ast}$ is away from $x_{2,\ast}$), which also has an extra factor of the large Carleman parameter $\lambda$.
Thus, by combining two Carleman estimates, we can absorb all non-positive terms into positive ones.

\begin{remark}
Similar tricks involving summing two Carleman estimates with different weights were used in \cite{AS, Jena, ArickS}, in the context of wave equations.
\end{remark}

Combining all the above leads to our Carleman estimate, for which an informal simplified version is stated below; see Theorem \ref{T.preciseC} for the precise statement.

\begin{theorem}[Global Carleman estimate]\label{T.mainglobal}
Let $\Omega \subseteq \R^n$ be a bounded domain, with a convex, connected, $C^4$-boundary $\Ga$, and fix $\sigma \in ( -\frac{3}{4}, 0 )$.
Then, there are two boundary defining functions $y_1, y_2 \in C^4(\Om)$, such that for any $T > 0$ and $\la \gg 1$, and for any sufficiently regular function $u$ satisfying $u = 0$ on $(0, T) \times \Ga$,
\begin{align}
\label{E.mainglobal} &C \la \int_{ (0, T) \times \Ga } (\mc{N}_\si u)^2 \, dS dt + \sum_{j=1,2} \int_{ (0,T) \times \Om } e^{-2 \la F_j} (\pm \pd_t u + \De_\si u)^2 \\
\notag &\quad \gtrsim \la\sum_{j=1,2} \int_{(0,T)\times \Om} e^{-2 \la F_j} \big[ y_j^{2\ka} |\nab u|^2 + y_j^{6\ka-1} (\la^2 + y_j^{-1-4\ka})u^2 \big] \text{,}
\end{align}
where $F_j$ is the Carleman weight \eqref{weight}, but with $y$ replaced by $y_j$.
\end{theorem}

The proof of Theorem \ref{T.mainglobal} follows the usual multiplier approach to Carleman estimates for heat equations, using the weights \eqref{weight} for both $y_1$ and $y_2$.
Aside from the ideas mentioned before, there are two key technical challenges to overcome.
The first is showing that the boundary terms capture the Neumann trace; this follows from computations for the boundary terms (Lemma \ref{L.pointC}), along with the detailed understanding of boundary asymptotics gained from Proposition \ref{neumann_trace}.

The second, and more difficult, challenge is to ensure all the top-order bulk terms obtained in the computations have good sign.
As there are many singular weights involved, we have more dangerous terms to consider than in standard derivations of Carleman estimates.
These are treated via extensive computations (see Lemma \ref{L.pointC}) that use, in an essential way, both the geometry of the domain---via convexity of the level hypersurfaces of $y$---and our assumption that $\si \in (-\frac{3}{4}, 0)$.

\begin{remark}
In the full statement, Theorem \ref{T.preciseC}, of our Carleman estimate, the power $1 + 2 \ka$ in \eqref{weight} is replaced by a more general parameter $2p$.
For purposes of boundary control, one requires $2p = 1 + 2 \ka$ to capture the Neumann trace.
However, allowing for more general powers $p$ leads to unique continuation properties for a larger range of $\sigma$.
We plan to revisit this point in a future paper.
\end{remark}

\begin{remark}
Note the estimate in the precise Theorem \ref{T.preciseC} differs from that of Theorem \ref{T.mainglobal} in that the Neumann integral in \eqref{E.mainglobal} is replaced by various boundary limits of integrals over hypersurfaces $\{ y_j = \delta \}$, as $\delta \searrow 0$.
However, one can show that, in the context of Problem (O), each of these boundary limits will either vanish or be bounded by the desired Neumann integral; for details, see the proof of Theorem \ref{obs}, as well as Section \ref{sec.neumann} for the limit computations.
The simpler Neumann integral was written in \eqref{E.mainglobal} for conceptual clarity.
\end{remark}

\subsection{Further discussions}

Let us now elaborate on the specific range $\sigma \in ( -\frac{3}{4}, 0 )$ that is assumed in all our main results.
As mentioned before, this is required in the proof of Theorem \ref{T.mainglobal} to ensure positivity of the bulk $L^2$-terms.
However, there are also conceptual reasons for applying this particular restriction.

First, the condition $\sigma > -\frac{3}{4}$ is crucial to the setup of our well-posedness theories.
As mentioned before, on the control side (Problem (C)), this is needed for solutions of \eqref{heat_ctl} with inhomogeneous Dirichlet data to be $L^2$-integrable on $\Om$.
Furthermore, on the observability side (Problem (O)), this seems necessary in order to bound the Neumann trace from above by the $H^1$-energy; see Proposition \ref{neumann_trace}.
The latter is an essential part of the Hilbert uniqueness method setup we apply here.
Thus, we do not expect our results to extend to $\sigma \leq -\frac{3}{4}$, at least within the well-posedness and HUM settings adopted in this paper.

Of course, the case $\si=0$ is just the classical heat equation, for which global Carleman estimates are now standard.
However, one should note that the proof of Theorem \ref{T.mainglobal} does not carry over to this case simply by setting $\si=0$, as it uses in a crucial way the critical singularity of the potential.

On the other hand, it is less clear whether our results can be extended beyond to the range $\si \in ( 0, \frac{1}{4} )$, though there seem to be some obstacles.
For one, note the Carleman estimate \eqref{E.mainglobal} fails to control the full $H^1$-energy when $\si > 0$ ($\ka > 0$).
Furthermore, in this regime, the boundary condition $u = 0$ in \eqref{heat_obs} does not directly imply $\mc{D}_\sigma u = 0$.
Therefore, one could expect that our Carleman weight \eqref{weight} and our choice of spaces are not well-adapted to this case $\si \in ( 0, \frac{1}{4} )$.

\begin{remark}
Also worth mentioning is the result of Gueye \cite{Gueye}, which established boundary controllability of the degenerate parabolic equation \eqref{degen} in one spatial dimension using spectral theoretic methods and a variant of Ingham's inequality.
However, this result cannot be directly compared to ours, since \cite{Gueye} uses different spaces in its HUM setup.
In particular, \cite{Gueye} showed that the $L^2$-norm of the Neumann datum controls the fractional $H^{-\ka}$-norm of the solution, and vice versa.
In contrast, we are less optimal with regards to regularity, but we use the smoothing property of parabolic equations to our advantage.
\end{remark}

Finally, for wave operators having the same singular potential (i.e., $-\partial_{t}^2 + \Delta_\sigma$), we recently established in \cite{JEMS}---in the special case $n > 2$ and $\Om$ a unit ball---boundary observability through a similar sharp global Carleman estimate.
While the Carleman weight is different from \eqref{weight}, due to the equation being hyperbolic, it is built upon the same sharp power of $d_\Ga$ yielding both the $H^1$-energy and the Neumann boundary trace.
Interestingly, this observability fails to imply boundary control for wave operators, as the lack of smoothing prevents us from applying the HUM machinery.
Using this framework, boundary controllability would necessitate working with fractional Sobolev spaces of optimal regularity, as in \cite{Gueye}.

One can also view Theorem \ref{T.mainglobal} partly as extending the methods of \cite{JEMS} (which hold only for $\Om$ being a unit ball) to all convex domains.
It would be interesting to determine whether the results of \cite{JEMS} also hold for general convex $\Om$.

\subsection{Organization of the paper}
	
In Section \ref{S.Carleman}, we construct boundary defining functions that coincide with the distance $d_\Ga$ near the boundary $\Ga$.
These are then used to prove a precise version of our global Carleman estimate, Theorem \ref{T.mainglobal}.
The applications of this Carleman estimate to boundary observability and controllability are then presented in Sections \ref{S.obs} and \ref{S.control}, respectively.	

\section{The Carleman estimate} \label{S.Carleman}

In this section, we prove a precise version of Theorem \ref{T.mainglobal}---our main Carleman estimate for parabolic operators with inverse square potentials.

In the remainder of the paper, we adopt the setting described in the beginning of the introduction---in particular the domain $\Omega \subseteq \R^n$, its convex boundary $\Ga$, and the distance $d_\Ga$ to the boundary.
Moreover, since $d_\Ga$ is always $C^4$ in a neighborhood of $\Ga$, we can also adopt the following for convenience:

\begin{setting}[Regularity of $d_\Ga$] \label{ass.d0}
Let $0 < d_0 \ll 1$ be such that $d_\Ga$ is $C^4$ on the domain
\[
\{ x \in \Om \mid d_\Ga (x) < 2 d_0 \} \text{.}
\]
\end{setting}

\subsection{Construction of boundary defining functions}

As described in the introduction, the proof of our Carleman estimate will require, as weights, boundary defining functions that extend $d_\Ga$ while essentially preserving concavity and regularity.
Here, we detail the construction of such functions.

First, we list the precise conditions needed for our constructions:

\begin{definition} \label{ep-bdf}
Given constants $\ep, \ep' > 0$, we call $y \in C^4 (\Om)$ an \emph{$( \ep, \ep' )$-boundary defining function} for $\Om$ iff the following properties hold:
\begin{enumerate}[\hspace{0.3cm}$a)$]
\item $y$ is strictly positive on $\Om$, and $y = d_\Gamma$ on $\{ x \in \Omega : d_\Gamma (x) < d_0 \}$.

\item $y$ has a unique critical point $x_\ast \in \Omega$, with $d_\Gamma (x_\ast) > 2 d_0$.

\item $y$ satisfies the following gradient bounds:
\begin{equation}
\label{est-grad} \begin{cases}
| \nab y |^2 = 1 &\quad d_\Gamma (x) \leq d_0 \text{,} \\
| \nab y |^2 \geq \tfrac{1}{2} &\quad d_0 < d_\Gamma < 2 d_0 \text{.}
\end{cases}
\end{equation}

\item $y$ satisfies the following concavity properties for each $x \in \Omega$ and $\xi \in \R^n$:
\begin{equation}
\label{est-pair} \begin{cases}
- \xi \cdot \nab^2 y (x) \cdot \xi \geq 0 &\quad d_\Gamma (x) \leq d_0 \text{,} \\
- \xi \cdot \nab^2 y (x) \cdot \xi \geq - \ep' | \xi |^2 &\quad d_0 < d_\Gamma (x) < 2 d_0 \text{,} \\
- \xi \cdot \nab^2 y (x) \cdot \xi \geq \ep |\xi|^2 &\quad d_\Gamma (x) \geq 2 d_0 \text{.}
\end{cases}
\end{equation}
\end{enumerate}
\end{definition}

\begin{remark}
Definition \ref{ep-bdf} implies $x_*$ is a non-degenerate maximum of $y$, that is, $\nab y ( x_* ) = 0$ and $\nab^2 y ( x_* )$ is negative-definite.
Furthermore, note that $x_*$ is the only maximum of $y$, so that $\nab y$ vanishes only at $x_*$.
\end{remark}

\begin{definition} \label{ep-pair}
Given any $\ep, \ep' > 0$, we refer to $( y_1, y_2 )$ as an \emph{$( \ep, \ep' )$-boundary defining pair} in $\Om$ iff the following properties hold:
\begin{enumerate}[$i)$]
\item Both $y_1$ and $y_2$ are $( \ep, \ep' )$-boundary defining functions.

\item The (unique) critical points of $y_1$ and $y_2$ are distinct: $x_{ 1, * } \neq x_{ 2, * }$.
\end{enumerate}
\end{definition}

In the proof of our global Carleman estimates, we will employ a carefully constructed boundary defining pair.
As the first step, we show that any convex domain admits such a pair, given sufficiently small parameters:

\begin{lemma} \label{L.convex}
There exist $C, C', \ep_0 > 0$---depending only on $\Omega, d_0$---such that for any $0 < \ep < \ep_0$, there exists a $( C \ep, C' \ep )$-boundary defining pair $(y_1, y_2)$ in $\Om$.
\end{lemma}

\begin{proof}
We begin by constructing one such boundary defining function $y_1$.
First, note that if $T_p \Ga$ is the tangent hyperplane to $\Ga$ at a point $p$ and $x \in \Om$, then
\[
d_\Ga (x) = \inf_{p\in\Ga} \operatorname{dist} (x, T_p \Ga) \text{,}
\]
by the convexity of $\Ga$.
This implies $d_\Ga$ is a concave function on $\Om$; in particular, for any $\xi \in \R^n$, the distributional derivative $- \xi \cdot \nab^2 d_\Ga \cdot \xi$ is a nonnegative measure.

Consider now the function
\begin{equation}
\label{d_moll} d^\ep := \phi^\ep \ast d_\Ga \text{,} \qquad \phi^\ep (x) := \ep^{-n} \phi \big( \tfrac{x}{\ep} \big)
\end{equation}
for a small parameter $0 < \ep < \ep_0$, where $\phi$ is a standard positive mollifier:
\[
\phi \in C^\infty_c ( B_1(0), [0,\infty) ) \text{,} \qquad \int_{\R^n} \phi(x) \, dx = 1 \text{.}
\]
Note $d^\ep$ is smooth and concave by the concavity of $d_\Ga$---indeed, for all $\xi \in \R^n$,
\begin{equation}
\label{d_moll_ccv} - \xi \cdot \nab^2 d^\ep \cdot \xi = \phi^\ep * ( - \xi \cdot \nab d_\Ga \cdot \xi ) \geq 0 \text{,}
\end{equation}
since $\phi \geq 0$ and $- \xi \cdot \nab^2 d_\Ga \cdot \xi$ is a nonnegative measure on $\Om$.
	
We now introduce a new cutoff $\vp \in C^\infty ( \Om )$ such that
\begin{equation}
\label{bdf_vp} \vp (x) = \begin{cases}
1 &\quad d_\Ga (x) \leq d_0 \text{,} \\
0 &\quad d_\Ga (x) \geq 2 d_0 \text{.}
\end{cases}
\end{equation}
In terms of these functions, we then set $y_1 \in C^4 ( \Omega )$ as
\begin{equation}
\label{bdfy} y_1 (x) := d_\Ga (x) + [1 - \vp(x)] [ d^\ep(x) - d_\Ga(x) - \ep |x|^2 ] \text{,}
\end{equation}
with $\ep < \ep_0 \ll d_0$.
(Note $1 - \vp \equiv 0$ near $\Ga$, where $d^\ep$ fails to be defined.)

To begin with, in the region $\{ d_\Ga \leq d_0 \}$, we note that $y_1 = d_\Ga$, since $\vp = 1$ there by \eqref{bdf_vp}.
Furthermore, the above implies
\[
| \nab y_1 |^2 = 1 \text{,} \qquad - \xi \cdot \nab^2 y_1 \cdot \xi \geq 0 \text{,} \quad \xi \in \R^n \text{.}
\]
In particular, all the conditions in Definition \ref{ep-bdf} are satisfied by $y_1$ on $\{ d_\Ga \leq d_0 \}$.

Next, consider the intermediate region $\{ d_0 < d_\Ga < 2 d_0 \}$.
Since $d_\Ga$ is $C^4$ there,
\[
\| d^\ep - d_\Ga \|_{ C^4 ( \{ d_0 < d_\Ga < 2 d_0 \} ) } \leq C \ep \text{.}
\]
Moreover, since $|\nab d_\Ga|=1$ and $|1-\vp| \leq 1$ in this region, the above implies
\begin{align}
\label{est-inter} y_1 &\geq \tfrac{d_0}{2} - C' \ep > 0 \text{,} \\
\notag |\nab y_1|^2 &\geq 1 - C' \ep^2 > \tfrac{1}{2} \text{,}
\end{align}
for sufficiently small $\ep_0$ (depending on $\Omega$, $d_0$), as well as
\begin{align}
\label{est-inter-2} - \xi \cdot \nab^2 y_1 \cdot \xi &= - \xi \cdot \nab^2 d_\Ga \xi - \xi \cdot \nab^2 \big[ (1-\vp) ( d^\ep - d_\Ga - \ep|x|^2 ) \big] \cdot \xi \\
\notag &\geq -C' \ep |\xi|^2
\end{align}
for all $\xi \in \R^n$, where $C' > 0$ denotes constants (depending on $\Om$, $d_0$) that can change between lines.
Thus, $y_1$ satisfies the conditions of Definition \ref{ep-bdf} on $\{ d_0 < d_\Ga < 2 d_0 \}$.

Lastly, consider the region $\{ d_\Ga \geq 2 d_0 \}$, on which
\[
y_1 = d^\ep - \ep |x|^2 \text{.}
\]
The above, along with \eqref{d_moll}, implies that $y_1$ is uniformly positive on this region for sufficiently small $\ep_0$, and that $y_1$ is uniformly concave, since by \eqref{d_moll_ccv},
\[
- \xi \cdot \nab^2 y_1 \cdot \xi = - \xi \cdot \nab^2 d^\ep \cdot \xi + 2 \ep |\xi|^2 \geq 2 \ep |\xi|^2 \text{,} \qquad \xi \in \R^n \text{.}
\]
Moreover, since $y_1$ is a positive function on $\Om$ whose gradient does not vanish on $\{ d_\Ga \leq 2 d_0 \}$ and which is uniformly concave on $\{ d_\Ga \geq 2 d_0 \}$, then $y_1$ must have a unique critical point $x_{1,*} \in \{ d_\Ga > 2 d_0 \}$, its maximum.
Therefore, $y_1$ satisfies all the conditions of Definition \ref{ep-bdf} on $\{ d_0 \geq 2 d_0 \}$.

The above yields that $y_1$ is a $( C \ep, C' \ep )$-boundary defining function, for appropriate constants $C, C'$ and sufficiently small $\ep_0$.
It remains to construct a new function $y_2$ so that $(y_1, y_2)$ defines a corresponding boundary defining pair.

To this end, note by the Morse lemma, there is a neighborhood $U \subset \{ d_\Ga > 2 d_0\}$ around $x_{1,*}$ and local coordinates $z: U \to\R^n$ such that $y_1$ is a quadratic form on $U$: $y_1 = z \cdot A \cdot z$, with $A$ a non-singular $n \times n$ matrix.
Furthermore, without loss of generality, one can assume $z (x_{1,*}) = 0$, and $U$ is an open ball $B_{2\eta}(0)$ in $z$-coordinates, for some small $\eta>0$.
As the critical point $x_{1,*}$ is non-degenerate, $|\nab y_1| \geq c_0 > 0$ on $U \setminus U'$, where $U' = B_\eta(0)$ in $z$-coordinates.
	
We then take a cutoff function $\chi \in C^\infty ( \Omega )$ satisfying $\chi \equiv 1$ on $U'$ and $\chi \equiv 0$ on on $\Om \setminus U$, and we define the function
\begin{equation}
\label{y2} y_2 := y_1 + \de \chi b \cdot z \text{,} \qquad b \in \Sph^{n-1} \text{,} \quad \de \ll 1 \text{.}
\end{equation}
Since $y_2$ coincides with $y_1$ on $\Om \setminus U$, then $y_2$ satisfies all the conditions of Definition \ref{ep-bdf} (with the same parameters) on both $\{ d_\Ga \leq d_0 \}$ and $\{ d_0 < d_\Ga < 2 d_0 \}$.

For the remaining inner region, note we have on $U'$ that
\[
\nab_z y_2 = 2 A \cdot z + \de b \text{,}
\]
so that $y_2$ has a critical point at
\[
z (x_{2,*}) = -\tfrac{1}{2} \de A^{-1} b \neq 0 \text{,}
\]
which is unique in $U'$.
Moreover, as the shift $\de\chi b\cdot z$ is supported in $U$ and
\begin{equation}
\label{shift-est} \| \de \chi b \cdot z \|_{C^j} \lesssim \de \eta^{1-j} \text{,} \qquad j \leq 4
\end{equation}
no new critical points are introduced as long as $\de$ is taken small enough with respect to $\eta$; in particular, we can ensure that $| \nab_z y_2 | > 0$ in $U \setminus U'$.
Similarly, by further shrinking $\delta$ if needed, \eqref{y2} and \eqref{shift-est} also ensure that on $\{ d_\Ga > 2 d_0 \}$,
\[
- \xi \cdot \nab^2 y_2 \cdot \xi \geq C \ep |\xi|^2 \text{,} \qquad \xi \in \R^n \text{,}
\]
possibly with a different value of $C$.
Thus, $y_2$ satisfies all the conditions of Definition \ref{ep-bdf}, hence $(y_1, y_2)$ is our desired boundary defining pair.
\end{proof}

Finally, since $d_\Ga$ fails to be regular away from $\Ga$, it will often be useful to replace our singular operator by a smoother variant:

\begin{definition} \label{def-operator'}
Given any $y \in C^4 ( \Omega )$, we define the $y$-modified operators
\begin{equation}\label{operator'}
\De_{ \sigma, y } := \De + \sigma y^{-2} \text{,} \qquad \sigma \in \R \text{.}
\end{equation}
For convenience, we also adopt the following notation for the $y$-derivative:
\begin{equation}
\label{Dyv} D_y v := \nab y \cdot \nab v \text{,} \qquad v \in C^1 ( \Omega ) \text{.}
\end{equation}
\end{definition}

In particular, when $y$ is a boundary defining function, so that $y = d_\Ga$ near $\Ga$, the difference $d_\Ga^{-2} - y^{-2}$ is hence bounded on $\Om$.
Thus, it will suffice to prove Carleman estimates for $\pm \partial_t + \Delta_{ \sigma, y }$, which has a $C^4$ singular potential, rather than $\pm \partial_t + \Delta_\sigma$.
	
With this notation in hand, we prove a pointwise Hardy-type inequality associated with $D_y$-derivatives that will useful in proving our Carleman estimates:

\begin{lemma} \label{pointHardy}
The following holds for any $q \in \R$, $y \in C^4 (\Om)$, and $v \in C^1(\Om)$:
\begin{align}
\label{Hardy} y^{2q} ( D_y v )^2 &\geq \tfrac{1}{4} (1-2q)^2 y^{2q-2} |\nab y|^4 \, v^2 + \nab \cdot \big[ \tfrac{1}{2} (1-2q) y^{2q-1} \nab y |\nab y|^2 \, v^2 \big] \\
\notag &\qquad - \tfrac{1}{2} (1-2q) \big[	y^{2q-1} \De y |\nab y|^2 \, v^2 + 2 y^{2q-1} ( \nab y \cdot \nab^2 y \cdot \nab y ) \, v^2 \big] \text{.}
\end{align}
\end{lemma}

\begin{proof}
This is a direct consequence of the inequality
\begin{align*}
0 &\leq ( y^q \, D_y v + b y^{ q - 1 } |\nab y|^2 \, v )^2 \\
&= y^{ 2 q } \, ( D_y v )^2 + b^2 y^{ 2 q - 2 } |\nab y|^4 \, v^2 + 2 b y^{ 2 q - 1 } |\nab y|^2 \, v D_y v \\
&= y^{ 2 q } \, ( D_y v )^2 + b ( b - 2 q + 1 ) y^{ 2 q - 2 } |\nab y|^4 \, v^2 - 2 b y^{ 2 q - 1 } ( \nab y \cdot \nab^2 y \cdot \nab y ) \, v^2	\\
&\qquad - b y^{ 2 q - 1 } \De y |\nab y|^2 \, v^2 + \nabla \cdot ( b y^{ 2q - 1 } \nabla y |\nab y|^2 \, v^2 ) \text{,}
\end{align*}
which holds for any constants $q, b \in \R$.
Equation~\eqref{Hardy} is then obtained by taking the optimal value
of the parameter $b := \tfrac{ 2q - 1 }{2}$.
\end{proof}

\subsection{The global pointwise inequality}

Our aim here is to prove the following key lemma, which serves as the pointwise Carleman inequality obtained from a single boundary defining function $y$.
In particular, this attains adequate control for our solutions everywhere except near the critical point of $y$.

\begin{lemma} \label{L.pointC}
Fix $T > 0$, and let $p, \sigma \in \R$ satisfy
\begin{equation}
\label{mu_kappa} 0 < p < \tfrac{1}{2} \text{,} \qquad p^2 - 2p + \sigma \geq -\tfrac{3}{4} \text{.}
\end{equation}
Let $\ep, \ep', \de > 0$ be sufficiently small (depending on $T, \Omega, d_0, \sigma, p$), let $y \in C^4 ( \Omega )$ be an $( \ep, \ep' )$-boundary defining function, and let $x_*$ denote the critical point of $y$.
Then, there exist $C, C', \lambda_0 > 0$ (depending on $T, \Omega, d_0, \sigma, p, \ep, \ep', \de$) such that the following inequality holds on $[0,T] \times \Om$ for any $\lambda \geq \lambda_0$ and any $u \in C^2 ( [0,T] \times \Om )$,
\begin{align}
\label{pointC} &e^{-2 \la F} | ( \pm \partial_t + \Delta_{ \sigma, y } ) u |^2 - 4 ( \pd_t J^t + \Div J ) \\
\notag &\quad \geq C \la \te e^{-2\la F} y^{-1+2p} \, | \nab u |^2 - C' \la^2 \te^3 e^{-2 \la F} y^{-3+4p} \mathbbm{1}_{ B_\delta (x_*) } \, u^2 \\
\notag &\quad\qquad + C e^{-2\la F} ( \la^3 \te^3 y^{-4+6p} + \la \te y^{-3+2p} ) \mathbbm{1}_{ \Om \setminus B_\delta (x_*) } \, u^2 \text{,}
\end{align}
where the weight $F$ is given by
\begin{equation}
\label{GF} F(t,x) := \te(t) f(y(x)) \text{,} \qquad \te(t) := \tfrac{1}{t(T-t)} \text{,} \qquad f(y) := \tfrac{1}{2p} y^{2p} + \be \text{,}
\end{equation}
with $\be > 0$ an arbitrary constant, where $J^t$ is a scalar function satisfying
\begin{equation}
\label{Jt-statement} | J^t | \leq C e^{-2 \lambda F} \, | \nab u |^2 + C e^{-2 \lambda F } \, \lambda^2 \theta^2 y^{-2} \, u^2 \text{,}
\end{equation}
and where $J$ is a vector field satisfying, sufficiently near $\Ga$,
\begin{align}
\label{J-statement} \nab y \cdot J - e^{-2 \lambda F } \, \partial_t u D_y u &\leq C e^{-2 \lambda F} \, \lambda \theta y^{-1+2p} \, ( D_y u )^2 \\
\notag &\qquad + C e^{-2 \lambda F} \, \lambda^3 \theta^3 y^{-3+2p} \, u^2 \text{.}
\end{align}
\end{lemma}

\begin{proof}
Throughout, we let $C, C' > 0$ denote constants with the same dependencies as in the lemma statement, and such that their values can change from line to line.
Furthermore, it suffices to prove \eqref{pointC} for just the backward operator $\partial_t + \Delta_{ \sigma, y }$, as the estimate for $- \partial_t + \Delta_{ \sigma, y }$ then follows via a time reversal $t \mapsto T - t$.

For clarity of exposition, we divide the proof into four steps.

\subsubsection*{Step 1: The conjugate inequality}

First, we prove the key preliminary commutator estimates for the operator $\partial_t + \Delta_{ \sigma, y }$.
For this, let us set
\begin{equation}
\label{v_G} v := e^{-\lambda F} u \text{.}
\end{equation}
Furthermore, the following constant will be useful later in the proof:
\begin{equation}
\label{eqz} z := \tfrac{1-2p}{2 \, y(x_*)} > 0 \text{.}
\end{equation}

Using \eqref{v_G} and \eqref{eqz}, we expand $( \partial_t + \Delta_{ \sigma, y } ) u$ as follows:
\begin{align}
\label{ident1} e^{-\lambda F} ( \partial_t + \Delta_{ \sigma, y } ) u &= e^{-\lambda F} ( \partial_t + \Delta_{ \sigma, y } ) ( e^{\lambda F} v ) \\
\notag &= Sv + \Delta v + \mc{A}_0 \, v \text{,}
\end{align}
where $Sv$ and $\mc{A}_0$ are given by
\begin{align}
\label{aux} Sv &:= \partial_t v + 2 \lambda \, \nabla F \cdot \nabla v + \lambda ( \Delta F - 2 z D_y F ) \, v \text{,} \\
\notag \mc{A}_0 &:= \lambda \partial_t F + 2 \lambda z D_y F + \lambda^2 | \nabla F |^2 + \sigma y^{-2} \text{.}
\end{align}
Multiplying \eqref{ident1} by $Sv$, and noting from Cauchy's inequality that
\[
e^{-\lambda F} ( \partial_t + \Delta_{ \sigma, y } ) u \, Sv \leq \tfrac{1}{4} e^{-2 \lambda F} | ( \partial_t + \Delta_{ \sigma, y } ) u |^2 + | Sv |^2 \text{,}
\]
we then conclude
\begin{equation}
\label{ident2} \tfrac{1}{4} e^{-2 \lambda F} | ( \partial_t + \Delta_{ \sigma, y } ) u |^2 \geq \Delta v S v + \mc{A}_0 \, v Sv \text{.}
\end{equation}

We now expand the terms on the right-hand side of \eqref{ident2}.
First, we have
\begin{align}
\label{I_Delta_0} \Delta v S v &= \Delta v \partial_t v + 2 \lambda \Delta v ( \nabla F \cdot \nabla v ) + \lambda ( \Delta F - 2 z D_y F ) \, v \Delta v \\
\notag &= I^\Delta_t + I^\Delta_1 + I^\Delta_0 \text{.}
\end{align}
The first term on the right-hand side is straightforward:
\begin{equation}
\label{I_Delta_1} I^\Delta_t = \nabla \cdot ( \nabla v \partial_t v ) + \partial_t \big( \tfrac{1}{2} | \nabla^2 v | \big) \text{.}
\end{equation}
The most involved term is $I^\Delta_1$, requiring multiple applications of the Leibniz rule:
\begin{align}
\label{I_Delta_2} I^\Delta_1 &= \nabla \cdot [ 2 \lambda \nabla v ( \nabla F \cdot \nabla v ) ] - \lambda \, \nabla F \cdot \nabla ( | \nabla v |^2 ) - 2 \lambda \, ( \nabla v \cdot \nabla^2 F \cdot \nabla v ) \\
\notag &= \nabla \cdot [ 2 \lambda \, \nabla v ( \nabla F \cdot \nabla v ) - \lambda \, \nabla F \, | \nabla v |^2 ] + \lambda \Delta F \, | \nabla v |^2 \\
\notag &\qquad - 2 \lambda \, ( \nabla v \cdot \nabla^2 F \cdot \nabla v ) \text{.}
\end{align}
A similar computation also yields
\begin{align}
\label{I_Delta_3} I^\Delta_0 &= \nabla \cdot [ \lambda ( \Delta F - 2 z D_y F ) \, v \nabla v ] - \lambda \, ( \Delta F - 2 z D_y F ) \, | \nabla v |^2 \\
\notag &\qquad - \tfrac{1}{2} \lambda \nabla ( \Delta F - 2 z D_y F ) \cdot \nabla ( v^2 ) \\
\notag &= \nabla \cdot \big[ \lambda ( \Delta F - 2 z D_y F ) \, v \nabla v - \tfrac{1}{2} \lambda \nabla ( \Delta F - 2 z D_y F ) \, v^2 \big] \\
\notag &\qquad - \lambda \, ( \Delta F - 2 z D_y F ) \, | \nabla v |^2 + \tfrac{1}{2} \lambda \Delta ( \Delta F - 2 z D_y F ) \, v^2 \text{.}
\end{align}
Moreover, for the remaining term in \eqref{ident2}, we expand
\begin{align}
\label{I_A} \mc{A}_0 \, v Sv &= \tfrac{1}{2} \mc{A}_0 \, \partial_t ( v^2 ) + \lambda \mc{A}_0 \, \nabla F \cdot \nabla ( v^2 ) + \lambda ( \Delta F - 2 z D_y F ) \mc{A}_0 \, v^2 \\
\notag &= \partial_t \big( \tfrac{1}{2} \mc{A}_0 \, v^2 \big) + \nabla \cdot ( \lambda \mc{A}_0 \nabla F \, v^2 ) - \tfrac{1}{2} \partial_t \mc{A}_0 \, v^2 \\
\notag &\qquad - \lambda \nabla F \cdot \nabla \mc{A}_0 \, v^2 - 2 z \lambda D_y F \mc{A}_0 \, v^2 \text{.}
\end{align}

Combining \eqref{I_Delta_0}--\eqref{I_A}, the estimate \eqref{ident2} then becomes
\begin{align}
\label{conjestimate} \tfrac{1}{4} e^{-2 \lambda F} | ( \partial_t + \Delta_{ \sigma, y } ) u |^2 &\geq \partial_t J^0_t + \nabla \cdot J^0 + 2 z \lambda D_y F \, | \nabla v |^2 \\
\notag &\qquad - 2 \lambda \, ( \nabla v \cdot \nabla^2 F \cdot \nabla v ) + \mc{A} \, v^2 \text{,}
\end{align}
where the zero-order coefficient $\mc{A}$ is given by
\begin{equation}
\label{coefA} \mc{A} := - \tfrac{1}{2} \partial_t \mc{A}_0 - \lambda \nabla F \cdot \nabla \mc{A}_0 - 2 z \lambda D_y F \mc{A}_0 + \tfrac{1}{2} \lambda \Delta ( \Delta F - 2 z D_y F ) \text{,}
\end{equation}
and where the scalar $J^t$ and vector field $J_0$ are given by
\begin{align}
\label{currents} J^t &= \tfrac{1}{2} \, | \nabla v |^2 + \tfrac{1}{2} \mc{A}_0 \, v^2 \text{,} \\
\notag J_0 &= \nabla v \partial_t v + 2 \lambda \, \nabla v ( \nabla F \cdot \nabla v ) - \lambda \, \nabla F \, | \nabla v |^2 + \lambda ( \Delta F - 2 z D_y F ) \, v \nabla v \\
\notag &\qquad - \tfrac{1}{2} \lambda \nabla ( \Delta F - 2 z D_y F ) \, v^2 + \lambda \mc{A}_0 \nabla F \, v^2 \text{.}
\end{align}

\subsubsection*{Step 2: First-order terms}

We record here the following identities for $F$:
\begin{align}
\label{hessianf} \nabla F &= \theta y^{-1+2p} \, \nabla y \text{,} \\
\notag \nabla^2 F &= -(1-2p) \theta y^{-2+2p} \, ( \nabla y \otimes \nabla y ) + \theta y^{-1+2p} \, \nabla^2 y \text{,} \\
\notag \Delta F &= -(1-2p) \theta y^{-2+2p} \, | \nabla y |^2 + \theta y^{-1+2p} \, \Delta y \text{.}
\end{align}
As a result, we see from \eqref{hessianf} that
\begin{align}
\label{first_order_1} &2 z \lambda D_y F \, | \nabla v |^2 - 2 \lambda \, ( \nabla v \cdot \nabla^2 F \cdot \nabla v ) \\
\notag &\quad = 2 z \lambda \theta y^{-1+2p} | \nabla y |^2 \, | \nabla v |^2 - 2 \lambda \theta y^{-1+2p} \, ( \nabla v \cdot \nabla^2 y \cdot \nabla v ) \\
\notag &\quad\qquad+ 2 (1-2p) \lambda \theta y^{-2+2p} \, ( D_y v )^2 \\
\notag &\quad \geq 2 \lambda \theta y^{-1+2p} \, [ \nabla v \cdot ( \eta z | \nabla y |^2 I - \nabla^2 y ) \cdot \nabla v ] + 2 (1-2p) \lambda \theta y^{-2+2p} \, ( D_y v )^2 \\
\notag &\quad\qquad + 2 (1-\eta) z \lambda \theta y^{-1+2p} \, ( D_y v )^2 \text{,}
\end{align}
for some $0 < \eta < 1$ whose value will be chosen later.

Applying the Hardy inequality (with $q = -1 + p$ and $q = -\frac{1}{2} + p$), we see that
\begin{align}
\label{first_order_2} &2 (1-2p) \lambda \theta y^{-2+2p} \, ( D_y v )^2 + 2 (1-\eta) z \lambda \theta y^{-1+2p} \, ( D_y v )^2 \\
\notag &\quad \geq \nabla \cdot J_H + \tfrac{1}{2} (1-2p) (3-2p)^2 \lambda \theta y^{-4+2p} | \nabla y |^4 \, v^2 \\
\notag &\qquad + 2 (1-\eta) (1-p)^2 z \lambda \theta y^{-3+2p} | \nabla y |^4 \, v^2 \\
\notag &\qquad - (1-2p) (3-2p) \lambda \theta y^{-3+2p} \Delta y | \nabla y |^2 \, v^2 - C' \lambda \theta y^{-2+2p} \, v^2 \text{,}
\end{align}
where the vector field $J_H$ is given by
\begin{align}
\label{current_hardy} J_H &:= (1-2p) (3-2p) \lambda \theta y^{-3+2p} \nabla y | \nabla y |^2 \, v^2 \\
\notag &\qquad + 2 (1-\eta) (1-p) z \lambda \theta y^{-2+2p} \nabla y | \nabla y |^2 \, v^2 \text{.}
\end{align}
In particular, we have collected all terms of order $y^{-2+2p}$ or better into the final negative term in the right-hand side of \eqref{first_order_2}.
Furthermore, any term containing $\nabla y \cdot \nabla^2 y \cdot \nabla y$ can be included in this negative term by default, since by Definition \ref{ep-bdf}, both $| \nabla y |^2 = 1$ and $\nabla y \cdot \nabla^2 y \cdot \nabla y = 0$ in the region $d_\Ga \leq d_0$.

Combining now \eqref{conjestimate} with \eqref{first_order_1}--\eqref{current_hardy} yields
\begin{align}
\label{first_order_10} &\tfrac{1}{4} e^{-2 \lambda F} | ( \partial_t + \Delta_{ \sigma, y } ) u |^2 - ( \partial_t J^t + \nabla \cdot J ) \\
\notag &\quad \geq 2 \lambda \theta y^{-1+2p} \, [ \nabla v \cdot ( \eta z | \nabla y |^2 I - \nabla^2 y ) \cdot \nabla v ] \\
\notag &\quad\qquad + \tfrac{1}{2} (1-2p) (3-2p)^2 \lambda \theta y^{-4+2p} | \nabla y |^4 \, v^2 \\
\notag &\quad\qquad + 2 (1-\eta) (1-p)^2 z \lambda \theta y^{-3+2p} | \nabla y |^4 \, v^2 \\
\notag &\quad\qquad - (1-2p) (3-2p) \lambda \theta y^{-3+2p} \Delta y | \nabla y |^2 \, v^2 \\
\notag &\quad\qquad + \mc{A} \, v^2 - C' \lambda \theta y^{-2+2p} \, v^2 \text{,}
\end{align}
where the vector field $J$ in \eqref{J-statement} can now be given explicitly by
\begin{align}
\label{current_J} J &:= J_0 + J_H \\
\notag &= \nabla v \partial_t v + 2 \lambda \, \nabla v ( \nabla F \cdot \nabla v ) - \lambda \, \nabla F \, | \nabla v |^2 + \lambda ( \Delta F - 2 z D_y F ) \, v \nabla v \\
\notag &\qquad - \tfrac{1}{2} \lambda \nabla ( \Delta F - 2 z D_y F ) \, v^2 + (1-2p) (3-2p) \lambda \theta y^{-3+2p} \nabla y | \nabla y |^2 \, v^2 \\
\notag &\qquad + 2 (1-\eta) (1-p) z \lambda \theta y^{-2+2p} \nabla y | \nabla y |^2 \, v^2 + \lambda \mc{A}_0 \nabla F \, v^2 \text{.}
\end{align}

Recalling that $y$ satisfies Definition \ref{ep-bdf}, we then have, for any $\xi \in \R^n$,
\begin{align}
\label{symmat} \eta z |\nab y|^2 |\xi|^2 - \xi \cdot \nab^2 y \cdot \xi &\geq \begin{cases}
\eta z |\xi|^2 &\quad d_\Gamma \leq d_0 \text{,} \\
\big( \tfrac{1}{2} \eta z - \ep' \big) |\xi|^2 &\quad d_0 < d_\Gamma < 2 d_0 \text{,} \\
\ep |\xi|^2 &\quad d_\Gamma \geq 2 d_0 \text{.}
\end{cases}
\end{align}
In particular, letting $\ep'$ be sufficiently small, and choosing
\begin{equation}
\label{eta_symmat} \eta := \tfrac{ 4 \ep' }{ z } \in ( 0, 1 ) \text{,}
\end{equation}
we obtain from \eqref{symmat} that
\begin{equation}
\label{symmat1} \eta z |\nab y|^2 |\xi|^2 - \xi \cdot \nab^2 y \cdot \xi \geq C | \xi |^2 \text{.}
\end{equation}
Furthermore, the same concavity properties \eqref{est-pair} also yield the following in $\Om$:
\begin{align}
\label{laplacianbound} - \De y \geq \begin{cases}
  0 &\quad d_\Gamma \leq d_0 \text{,} \\
  -	\ep' n &\quad d_0 < d_\Gamma < 2 d_0 \text{,} \\
  \ep n &\quad d_\Gamma \geq 2 d_0 \text{.}
\end{cases}
\end{align}
From \eqref{first_order_10} and \eqref{symmat1}, we now conclude
\begin{align}
\label{first_order_20} &\tfrac{1}{4} e^{-2 \lambda F} | ( \partial_t + \Delta_{ \sigma, y } ) u |^2 - ( \partial_t J_t + \nabla \cdot J ) \\
\notag &\quad \geq C \lambda \theta y^{-1+2p} \, | \nabla v |^2 + \tfrac{1}{2} (1-2p) (3-2p)^2 \lambda \theta y^{-4+2p} | \nabla y |^4 \, v^2 \\
\notag &\quad\qquad + 2 (1-\eta) (1-p)^2 z \lambda \theta y^{-3+2p} | \nabla y |^4 \, v^2 \\
\notag &\quad\qquad - (1-2p) (3-2p) \lambda \theta y^{-3+2p} \Delta y | \nabla y |^2 \, v^2 \\
\notag &\quad\qquad + \mc{A} \, v^2 - C \lambda \theta y^{-2+2p} \, v^2 \text{.}
\end{align}

\subsubsection*{Step 3: The zeroth order terms}

It remains to estimate the zero-order coefficient $\mc{A}$.
First, for $\mc{A}_0$, note from \eqref{aux} and \eqref{hessianf} that
\begin{equation}
\label{zero_order_0} \mc{A}_0 = \sigma y^{-2} + \lambda^2 \theta^2 y^{-2+4p} | \nabla y |^2 + 2 z \lambda \theta y^{-1+2p} | \nabla y |^2 + \tfrac{1}{2p} \lambda \theta' y^{2p} \text{.}
\end{equation}
Using that
\begin{equation}
\label{theta_deriv} | \theta' | \lesssim \theta^2 \text{,} \qquad | \theta'' | \lesssim \theta^3 \text{,}
\end{equation}
we then compute that
\begin{equation}
\label{zero_order_1} - \tfrac{1}{2} \partial_t \mc{A}_0 \geq - C' \lambda^2 \theta^3 y^{-2+2p} \text{.}
\end{equation}

Next, using \eqref{hessianf} and \eqref{zero_order_0} we expand
\begin{align}
\label{zero_order_2} - \lambda \nabla F \cdot \nabla \mc{A}_0 &= - \lambda \theta y^{-1+2p} D_y \mc{A}_0 \\
\notag &\geq 2 \sigma \lambda \theta y^{-4+2p} | \nabla y |^2 + 2 (1-2p) \lambda^3 \theta^3 y^{-4+6p} | \nabla y |^4 \\
\notag &\qquad - \lambda^3 \theta^3 y^{-3+6p} ( \nabla y \cdot \nabla^2 y \cdot \nabla y ) \\
\notag &\qquad + 2 z (1-2p) \lambda^2 \theta^2 y^{-3+4p} | \nabla y |^4 - C' \lambda^2 \theta^2 y^{-2+2p} \text{,}
\end{align}
as well as
\begin{align}
\label{zero_order_3} - 2 z \lambda D_y F \, \mc{A}_0 &= - 2 z \lambda \theta y^{-1+2p} | \nabla y |^2 \, \mc{A}_0 \\
\notag &\geq - 2 \sigma z \lambda \theta y^{-3+2p} | \nabla y |^2 - 2 z \lambda^3 \theta^3 y^{-3+6p} | \nabla y |^4 \\
\notag &\qquad - C' \lambda^2 \theta^2 y^{-2+2p} \text{.}
\end{align}
Lastly, observe from \eqref{hessianf} that
\begin{align}
\label{zero_order_4} &\tfrac{1}{2} \lambda \Delta ( \Delta F - 2 z D_y F ) \\
\notag &\quad = \tfrac{1}{2} \lambda \theta \Delta [ -(1-2p) y^{-2+2p} \, | \nabla y |^2 + y^{-1+2p} \, \Delta y - 2 z y^{-1+2p} | \nabla y |^2 ] \\
\notag &\quad \geq - (1-2p) (1-p) (3-2p) \lambda \theta y^{-4+2p} | \nabla y |^4 \\
\notag &\quad\qquad + 2 (1-2p) (1-p) \lambda \theta y^{-3+2p} | \nabla y |^2 \Delta y \\
\notag &\quad\qquad - 2 z (1-2p) (1-p) \lambda \theta y^{-3+2p} | \nabla y |^4 - C' \lambda \theta y^{-2+2p} \text{.}
\end{align}

Thus, combining \eqref{coefA} and \eqref{zero_order_1}--\eqref{zero_order_4} yields
\begin{align}
\label{zero_order_10} \mc{A} &\geq [ 2 \sigma | \nabla y |^{-2} - (1-2p) (1-p) (3-2p) ] \lambda \theta y^{-4+2p} | \nabla y |^4 \\
\notag &\qquad - [ 2 \sigma z | \nabla y |^{-2} + 2 z (1-2p) (1-p) ] \lambda \theta y^{-3+2p} | \nabla y |^4 \\
\notag &\qquad + 2 (1-2p) (1-p) \Delta y | \nabla y |^{-2} \, \lambda \theta y^{-3+2p} | \nabla y |^4 \\
\notag &\qquad + 2 \lambda^3 \theta^3 [ (1-2p) - zy ] y^{-4+6p} | \nabla y |^4 \\
\notag &\qquad + 2 z (1-2p) \lambda^2 \theta^2 y^{-3+4p} | \nabla y |^4 - C' \lambda^2 \theta^3 y^{-2+2p} \text{.}
\end{align}
Putting \eqref{zero_order_10} together with our estimate \eqref{first_order_20}, we then have
\begin{align}
\label{zero_order_11} &\tfrac{1}{4} e^{-2 \lambda F} | ( \partial_t + \Delta_{ \sigma, y } ) u |^2 - ( \partial_t J_t + \nabla \cdot J ) \\
\notag &\quad \geq C \lambda \theta y^{-1+2p} \, | \nabla v |^2 + 2 \big( p^2 - 2p + \sigma | \nabla y |^{-2} + \tfrac{3}{4} \big) \lambda \theta y^{-4+2p} | \nabla y |^4 \, v^2 \\
\notag &\quad\qquad + 2 z [ p - \eta (1-p) - \sigma | \nabla y |^{-2} ] \lambda \theta y^{-3+2p} | \nabla y |^4 \, v^2 \\
\notag &\quad\qquad - (1-2p) \lambda \theta y^{-3+2p} \Delta y | \nabla y |^2 \, v^2 \\
\notag &\quad\qquad + 2 \lambda^3 \theta^3 [ (1-2p) - zy ] y^{-4+6p} | \nabla y |^4 \, v^2 \\
\notag &\quad\qquad + 2 z (1-2p) \lambda^2 \theta^2 y^{-3+4p} | \nabla y |^4 \, v^2 - C' \lambda^2 \theta^3 y^{-2+2p} \, v^2 \text{.}
\end{align}

Choosing $\varepsilon'$ sufficiently small, so that \eqref{eta_symmat} implies
\[
p - \eta (1-p) \geq \tfrac{1}{2} p + C \varepsilon \text{,}
\]
and recalling \eqref{mu_kappa}, \eqref{eqz}, and \eqref{laplacianbound}, we conclude that
\begin{align}
\label{zero_order_12} &\tfrac{1}{4} e^{-2 \lambda F} | ( \partial_t + \Delta_{ \sigma, y } ) u |^2 - ( \partial_t J_t + \nabla \cdot J ) \\
\notag &\quad \geq C \lambda \theta y^{-1+2p} \, | \nabla v |^2 + C \lambda^3 \theta^3 y^{-4+6p} | \nabla y |^4 \, v^2 \\
\notag &\quad\qquad + 2 \big( p^2 - 2p + \sigma | \nabla y |^{-2} + \tfrac{3}{4} \big) \lambda \theta y^{-4+2p} | \nabla y |^4 \, v^2 \\
\notag &\quad\qquad + z ( p - 2 \sigma | \nabla y |^{-2} ) \lambda \theta y^{-3+2p} | \nabla y |^4 \, v^2 - C' \lambda^2 \theta^3 y^{-3+4p} \, v^2 \text{.}
\end{align}
We now claim that \eqref{zero_order_12} implies
\begin{align}
\label{conjestimate2} &\tfrac{1}{4} e^{-2 \lambda F} | ( \partial_t + \Delta_{ \sigma, y } ) u |^2 - ( \partial_t J_t + \nabla \cdot J ) \\
\notag &\quad \geq C \lambda \theta y^{-1+2p} \, | \nabla v |^2 + C ( \lambda^3 \theta^3 y^{-4+6p} + \lambda \theta y^{-3+2p} ) | \nabla y |^4 \, v^2 \\
\notag &\quad\qquad - C' \lambda^2 \theta^3 y^{-3+4p} \, v^2 \text{.}
\end{align}

First, when $\sigma \leq 0$, then \eqref{conjestimate2} follows from \eqref{mu_kappa}, \eqref{zero_order_12}, and the inequalities
\begin{align*}
2 \sigma | \nabla y |^{-2} \, \lambda \theta y^{-4+2p} | \nabla y |^4 \, v^2 &\geq 2 \sigma \, \lambda \theta y^{-4+2p} | \nabla y |^4 \, v^2 - C' \lambda \theta y^{-2+2p} \, v^2 \text{,} \\
2 \big( p^2 - 2p + \sigma + \tfrac{3}{4} \big) \lambda \theta y^{-4+2p} | \nabla y |^4 \, v^2 &\geq 0 \text{,} \\
z ( p - 2 \sigma | \nabla y |^{-2} ) \lambda \theta y^{-3+2p} | \nabla y |^4 \, v^2 &\geq z p \, \lambda \theta y^{-3+2p} | \nabla y |^4 \text{.}
\end{align*}
(The first inequality above follows, since $| \nab y |^{-2} - 1$ vanishes near $\Ga$ and is bounded from below on $\Om$ by a negative constant.)
On the other hand, the case $\sigma > 0$ follows from \eqref{mu_kappa}, \eqref{eqz}, \eqref{zero_order_12}, and the inequalities
\begin{align*}
2 \big( p^2 - 2p + \tfrac{3}{4} \big) \lambda \theta y^{-4+2p} | \nabla y |^4 \, v^2 &\geq 0 \text{,} \\
2 \sigma \, \lambda \theta y^{-4+2p} | \nabla y |^2 \, v^2 - z \, 2 \sigma \, \lambda \theta y^{-3+2p} | \nabla y |^2 \, v^2 &\geq 4 p \sigma \, \lambda \theta y^{-4+2p} | \nabla y |^2 \, v^2 \text{.}
\end{align*}
Combining the above two cases completes the proof of \eqref{conjestimate2}.

\subsubsection*{Step 4: Completion of the proof}

Since $x_\ast$ is the only critical point of $y$, then $| \nab y |$ is bounded away from zero on $\Omega \setminus B_\delta ( x_* )$, and hence \eqref{conjestimate2} becomes
\begin{align*}
&e^{-2 \lambda F} | ( \partial_t + \Delta_{ \sigma, y } ) u |^2 - 4 ( \partial_t J_t + \nabla \cdot J ) \\
&\quad \geq C \lambda \theta y^{-1+2p} \, | \nabla v |^2 + C ( \lambda^3 \theta^3 y^{-4+6p} + \lambda \theta y^{-3+2p} ) \mathbbm{1}_{ \Omega \setminus B_\delta ( x_* ) } \, v^2 \\
&\quad \qquad - C' \lambda^2 \theta^3 y^{-3+4p} \, v^2 \text{.}
\end{align*}
Furthermore, on $\Omega \setminus B_\delta ( x_* )$, the negative term in the right-hand side of the above can be absorbed into the positive terms by taking $\lambda$ sufficiently large, and hence
\begin{align}
\label{conjestimate3} &e^{-2 \lambda F} | ( \partial_t + \Delta_{ \sigma, y } ) u |^2 - 4 ( \partial_t J_t + \nabla \cdot J ) \\
\notag &\quad \geq C \lambda \theta y^{-1+2p} \, | \nabla v |^2 + C ( \lambda^3 \theta^3 y^{-4+6p} + \lambda \theta y^{-3+2p} ) \mathbbm{1}_{ \Omega \setminus B_\delta ( x_* ) } \, v^2 \\
\notag &\quad\qquad - C' \lambda^2 \theta^3 y^{-3+4p} \mathbbm{1}_{ B_\delta ( x_* ) } \, v^2 \text{.}
\end{align}
The desired estimate \eqref{pointC} now follows from \eqref{conjestimate3} by \eqref{v_G}, and by noting that
\begin{align*}
e^{-2\la F} y^{-1+2p} \, |\nab u|^2 &\leq C y^{-1+2p} |\nab v|^2 + C \la^2 \te^2 y^{-3+6p} |\nab y|^2 v^2 \\
&\leq C y^{-1+2p} |\nab v|^2 + C \la^2 \te^2 y^{-4+6p} \mathbbm{1}_{ \Omega \setminus B_\delta ( x_* ) } v^2 \\
&\qquad + C' \la^2 \te^2 y^{-3+4p} \mathbbm{1}_{ B_\delta ( x_* ) } v^2 \text{.}
\end{align*}

It remains to prove the inequalities \eqref{Jt-statement} and \eqref{J-statement}.
For the latter bound, we apply \eqref{v_G}, \eqref{hessianf}, \eqref{current_J}, and \eqref{zero_order_0} to obtain
\begin{align*}
\nab y \cdot J - \partial_t v D_y v &\leq 2 \lambda \theta y^{-1+2p} \, ( D_y v )^2 - \lambda \theta y^{-1+2p} \, | \nab v |^2 \\
&\qquad + C \lambda \theta y^{-2+2p} \, |v| | D_y v | + C \lambda \theta y^{-3+2p} \, v^2 \\
&\leq C \lambda \theta y^{-1+2p} \, ( D_y v )^2 + C \lambda \theta y^{-3+2p} \, v^2 \\
&\leq C e^{-2 \lambda F} \, \lambda \theta y^{-1+2p} \, ( D_y v )^2 + C e^{-2 \lambda F} \, \lambda^3 \theta^3 y^{-3+2p} \, u^2 \text{,}
\end{align*}
whenever $d_\Ga < d_0$ (so that $| \nab y |^2 = 1$ by Definition \ref{ep-bdf}).
Notice, in particular, that terms containing derivatives of $v$ in directions other than along $\nab y$ are non-positive and hence can be omitted.
The desired \eqref{J-statement} now follows from the above, \eqref{GF}, \eqref{v_G}, and \eqref{theta_deriv}.
Similarly, for \eqref{Jt-statement}, we estimate
\begin{align*}
| J^t | &\leq C | \nab v |^2 + C \lambda^2 \theta^2 y^{-2} \, v^2 \\
&\leq C e^{-2 \lambda F} \, | \nab u |^2 + C e^{-2 \lambda F } \, \lambda^2 \theta^2 y^{-2} \, u^2 \text{,}
\end{align*}
where we applied \eqref{v_G}, \eqref{currents}, \eqref{hessianf}, and \eqref{zero_order_0}.
\end{proof}

\subsection{The global Carleman estimate}

In this subsection, we will state and prove the precise version our main global Carleman estimate.
Before doing so, we must first improve the pointwise estimate \eqref{pointC} by eliminating the negative term in the right-hand side that is supported near the critical point of the boundary defining function.
This is accomplished below by summing two instances of \eqref{pointC}, using two boundary defining functions with distinct critical points.
	
\begin{lemma} \label{L.gluing}
Fix $T > 0$, and let $\sigma, p \in \R$ satisfy \eqref{mu_kappa}.
There exist $C, \ep, \ep', \lambda_0 > 0$ (depending on $T, \Omega, d_0, \sigma, p$) and an $( \ep, \ep' )$-boundary defining pair $(y_1, y_2)$ such that the following holds for all $u \in C^2 ( [ 0, T ] \times \Om )$ and $\lambda \geq \lambda_0$,
\begin{align}
\label{Carleman00} &\sum_{ j = 1 }^2 e^{-2 \la F_j} | ( \pm \partial_t + \Delta_{ \sigma, y_j } ) u |^2 - 4 \sum_{j=1}^2 ( \pd_t J_j^t + \Div J_j ) \\
\notag &\quad \geq C \sum_{ j = 1 }^2 e^{ -2 \la F_j } \big[ \la \te y_j^{-1+2p} | \nab u |^2 + ( \la^3 \te^3 y_j^{-4+6p} + \la \te y_j^{-3+2p} ) u^2 \big] \text{,}
\end{align}
where $F_j$ ($j = 1, 2$) is given by
\begin{equation}
\label{weightmu} F_j(t,x) := \theta (t) \Big( \tfrac{1}{2p} y_j(x)^{2p} + \be_j \Big) \text{,} \qquad \theta (t) := \frac{1}{ t (T-t) } \text{,}
\end{equation}
for appropriately chosen $\beta_j > 0$, where the scalars $J_j^t$ satisfy
\begin{equation}
\label{Jt-statement-j} | J_j^t | \leq C e^{-2 \lambda F_j} \, | \nab u |^2 + C e^{-2 \lambda F_j} \, \lambda^2 \theta^2 y_j^{-2} \, u^2 \text{,}
\end{equation}
and where the vector fields $J_j$ satisfy, sufficiently near $\Ga$,
\begin{align}
\label{J-statement-j} \nab y_j \cdot J_j - e^{-2 \lambda F_j} \partial_t u D_{ y_j } u &\leq C e^{-2 \lambda F_j} \, \lambda \theta y_j^{-1+2p} \, ( D_{ y_j } u )^2 \\
\notag &\qquad + C e^{-2 \lambda F_j} \, \lambda^3 \theta^3 y_j^{-3+2p} \, u^2 \text{.}
\end{align}
\end{lemma}

\begin{figure}[t]
\begin{center}
\includegraphics[width=0.88\textwidth]{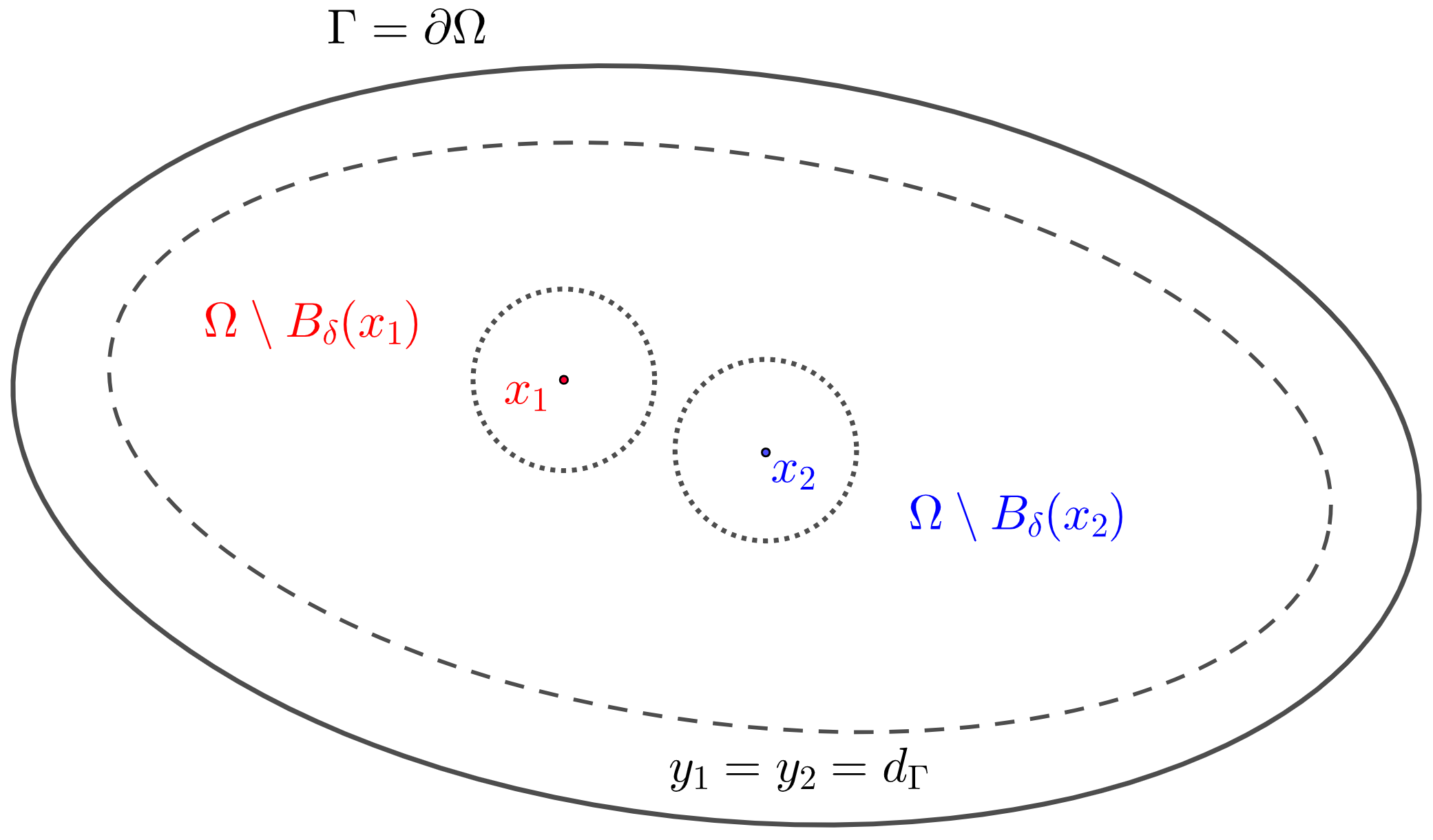}
\caption{The domain $\Om$ with convex boundary $\Ga$ is depicted together with balls centered at the critical points $x_1,x_2$ of two good boundary defining functions $y_1, y_2$.
In a neighborhood of $\Ga$, these functions agree with the distance to the boundary $d_\Ga$.}\label{fig.domain}
\end{center}
\end{figure}
	
\begin{proof}
Lemma \ref{L.convex} yields a $( \ep, \ep' )$-boundary defining $(y_1, y_2)$ satisfying the properties of Definition \ref{ep-pair}, for any sufficiently small $\ep$ and $\ep'$.
In particular, $y_j$ ($j = 1, 2$) has a unique critical point $x_j \in \Omega$, with $d_\Gamma ( x_j ) > 2 d_0$, at which it attains its maximum $R_j := y_j (x_j)$.
Since $x_j$ is the (unique) global maximum of $y_j$, and since $x_1 \neq x_2$, there exist $\delta > 0$, $0 < r_1 < R_1$, and $0 < r_2 < R_2$ such that
\begin{align}
\label{balls_disjoint} B_\delta ( x_1 ) \cap B_\delta ( x_2 ) &= \emptyset \text{,} \\
\notag \{ r_1 \leq y_1 \leq R_1 \} \cap \{ y_2 \leq r_2 \} &\supseteq B_\delta (x_1) \text{,} \\
\notag \{ r_2 \leq y_2 \leq R_2 \} \cap \{ y_1 \leq r_1 \} &\supseteq B_\delta (x_2) \text{.}
\end{align}
See Figure \ref{fig.domain} for an illustration of this setting.
		
We can thus apply Lemma \ref{L.pointC} with $y = y_j$, our given $p, \sigma$, the above $\delta$, and sufficiently large $\lambda > 0$.
Summing both estimates, we derive
\begin{align}
\label{pointCj} &\sum_{ j = 1 }^2 e^{-2 \la F_j } | ( \pm \partial_t + \Delta_{ \sigma, y_j } u ) |^2 - 4 \sum_{ j = 1 }^2 ( \pd_t J_j^t + \Div J_j ) \\
\notag &\quad \geq C \la \te \sum_{ j = 1 }^2 e^{-2 \la F_j} y_j^{-1+2p} | \nab u |^2 - C' \la^2 \te^3 \sum_{ j = 1 }^2 e^{-2 \la F_j } y_j^{-3+4p} \mathbbm{1}_{ B_\delta (x_j) } \, u^2 \\
\notag &\quad\qquad + C \sum_{ j = 1 }^2 e^{-2 \la F_j} ( \la^3 \te^3 y_j^{-4+6p} + \la \te y_j^{-3+2p} ) \mathbbm{1}_{ \Omega \setminus B_\delta (x_j) } \, u^2 \text{.}
\end{align}
Furthermore, we write the Carleman weights $F_j$ as
\[
F_j(t,x) = \te(t) f_j(y_j(x)) \text{,} \qquad f_j (r) := \tfrac{1}{2p} r^{2p} + \be_j \text{,}
\]
and we choose $\be_1, \be_2 > 0$ to satisfy
\[
\be_2 - \be_1 := \tfrac{1}{2p} ( r_1^{2p} - r_2^{2p} ) \text{.}
\]
Notice that with the above choice, we have
\begin{equation}
\label{lambdaj} f_1 (r_1) = f_2 (r_2) \text{.}
\end{equation}

Then, for each $j = 1, 2$ and $j^* := 3 - j$, we have that
\begin{align}
\label{crucial_absorb} e^{ -2 \la F_j } y_j^{ -3 + 4p } \mathbbm{1}_{ B_\delta ( x_j ) } &\leq e^{ -2 \la \theta (t) f_j ( r_j ) } y_j^{ -4 + 6p } \mathbbm{1}_{ B_\delta ( x_j ) } \\
\notag &\leq C e^{ -2 \la \theta (t) f_{j^*} ( r_{j^*} ) } y_{ j^* }^{ -4 + 6p } \mathbbm{1}_{ B_\delta ( x_j ) } \\
\notag &\leq C e^{ -2 \la F_{j^*} } y_{ j^* }^{ -4 + 6p } \mathbbm{1}_{ \Omega \setminus B_\delta ( x_{ j^* } ) } \text{,}
\end{align}
with $C > 0$ independent of $\lambda$.
(Here, the first and third steps in \eqref{crucial_absorb} follow from \eqref{balls_disjoint} and the monotonicity of $f_1$ and $f_2$, while the second step is a consequence of \eqref{lambdaj} and the fact that both $y_j$, $y_{ j^* }$ are bounded away from $0$ on $B_\delta ( x_j )$.)

Applying \eqref{crucial_absorb} and taking $\lambda$ large enough, the negative term in the right-hand side of \eqref{pointCj} can be absorbed in the subsequent positive term, and we arrive at
\begin{align*}
&\sum_{ j = 1 }^2 e^{-2 \la F_j} | ( \pm \partial_t + \Delta_{ \sigma, y_j } ) u |^2 - 4 \sum_{ j = 1 }^2 ( \pd_t J_j^t + \Div J_j ) \\
\notag &\quad \geq C \sum_{ j = 1 }^2 e^{-2 \la F_j} \big[ \la \te y_j^{-1+2p} | \nab u |^2 + ( \la^3 \te^3 y_j^{-4+6p} + \la \te y_j^{-3+2p} ) \mathbbm{1}_{ \Omega \setminus B_\delta (x_j) } \, u^2 \big] \text{.}
\end{align*}
Finally, the desired \eqref{Carleman00} follows by noting that the factor $\mathbbm{1}_{ \Omega \setminus B_\delta ( x_j ) }$ in the above can be removed---this is because $\Omega$ is covered by $\Omega \setminus B_\delta ( x_1 )$ and $\Omega \setminus B_\delta ( x_2 )$, and both $y_1$, $y_2$ are bounded away from zero on $B_\delta ( x_1 ) \cup B_\delta ( x_2 )$.
\end{proof}

We can now state our main global Carleman estimate:

\begin{theorem} \label{T.preciseC}
Fix $T > 0$, and let $\sigma, p \in \R$ satisfy \eqref{mu_kappa}.
Then, there exist constants $C, \ep, \ep', \lambda_0 > 0$ (depending on $T, \Omega, d_0, \sigma, p$) and an $( \ep, \ep' )$-boundary defining pair $(y_1, y_2)$ such that the following Carleman estimate holds,
\begin{align}
\label{intC} &C \sum_{j=1}^2 \int_{ [0,T] \times \Om } e^{-2 \la F_j} \big[ \lambda \theta y_j^{-1+2p} |\nab u|^2 + ( \la^3 \te^3 y_j^{-4+6p} + \la \te y_j^{-3+2p} ) u^2 \big] \\
\notag &\quad \leq \limsup_{ \delta \searrow 0 } \sum_{ j = 1 }^2 \int_{ [0,T] \times \{ y_j = \delta \} } e^{-2 \lambda F_j} [ \lambda \theta y_j^{-1+2p} \, ( D_{ y_j } u )^2 + \lambda^3 \theta^3 y_j^{-3+2p} \, u^2 ] \\
\notag &\quad\qquad + \limsup_{ \delta \searrow 0 } \sum_{ j = 1 }^2 \bigg| \int_{ [0,T] \times \{ y_j = \delta \} } e^{-2 \lambda F_j} \, \partial_t u D_{ y_j } u \bigg| \\
\notag &\quad\qquad + \sum_{j=1}^2 \int_{ [0,T] \times \Om } e^{-2 \la F_j} | ( \pm \partial_t + \Delta_{ \sigma, y_j } ) u |^2 \text{,}
\end{align}
for all $\lambda \geq \lambda_0$ and for all $u \in C^2 ( [ 0, T ] \times \Om )$ having finite energy,
\begin{equation}
\label{finite_energy} \sup_{ t \in [ 0, T ] } \int_{ \{t\} \times \Omega } ( | \nab u |^2 + d_\Ga^{-2} u^2 ) < \infty \text{,}
\end{equation}
and where both $F_j$ ($j = 1, 2$) and $\te$ are defined as in \eqref{weightmu}.
\end{theorem}

\begin{proof}
Let $C, \ep, \ep', \lambda_0, ( y_1, y_2 )$ be chosen as in Lemma \ref{L.gluing}.
Integrating the pointwise estimate \eqref{Carleman00} over the domain $[ 0, T ] \times \{ y_j > \delta \}$ and applying both the fundamental theorem of calculus (in $t$) and the divergence theorem (in $x$) yields
\begin{align}
\label{intC_0} &C \sum_{ j = 1 }^2 \int_{ [ 0, T ] \times \{ y_j > \delta \} } e^{ -2 \la F_j } \big[ \la \te y_j^{-1+2p} | \nab u |^2 + ( \la^3 \te^3 y_j^{-4+6p} + \la \te y_j^{-3+2p} ) u^2 \big] \\
\notag &\quad \leq \sum_{ j = 1 }^2 \int_{ [ 0, T ] \times \{ y_j > \delta \} } e^{-2 \la F_j} | ( \pm \partial_t + \Delta_{ \sigma, y_j } ) u |^2 \\
\notag &\quad\qquad + 4 \sum_{j=1}^2 \int_{ \{ T \} \times \{ y_j > \delta \} } | J_j^t | + 4 \sum_{j=1}^2 \int_{ \{ 0 \} \times \{ y_j > \delta \} } | J_j^t | \\
\notag &\quad\qquad + 4 \sum_{j=1}^2 \int_{ [ 0, T ] \times \{ y_j = \delta \} } ( \nabla y_j \cdot J_j ) \text{.}
\end{align}

By \eqref{J-statement-j}, there exists $C' > 0$ (with the same dependencies as before) with
\begin{align}
\label{intC_1} \int_{ [ 0, T ] \times \{ y_j = \delta \} } ( \nabla y_j \cdot J_j ) &\leq C' \int_{ [ 0, T ] \times \{ y_j = \delta \} } e^{-2 \lambda F_j} \, \lambda \theta y_j^{-1+2p} \, ( D_{ y_j } u )^2 \\
\notag &\qquad + C' \int_{ [ 0, T ] \times \{ y_j = \delta \} } e^{-2 \lambda F_j} \, \lambda^3 \theta^3 y_j^{-3+2p} \, u^2 \\
\notag &\qquad + C' \bigg| \int_{ [0,T] \times \{ y_j = \delta \} } e^{-2 \lambda F_j} \, \partial_t u D_{ y_j } u \bigg| \text{,}
\end{align}
for $j = 1, 2$.
For the remaining boundary integrals for $J^t_j$, note that
\[
\lambda^k \theta^k e^{-\lambda F_j} \leq \lambda^k \theta^k e^{-\la \theta \beta_j }
\]
converges uniformly to $0$ as $t \nearrow T$ and $t \searrow 0$, for any $k \geq 0$.
The above, combined with \eqref{Jt-statement-j} and \eqref{finite_energy}, imply that the terms of \eqref{intC_0} containing $J^t_1$, $J^t_2$ vanish:
\begin{equation}
\label{intC_2} 4 \sum_{j=1}^2 \int_{ \{ T \} \times \{ y_j < \delta \} } | J_j^t | = 4 \sum_{j=1}^2 \int_{ \{ 0 \} \times \{ y_j < \delta \} } | J_j^t | = 0 \text{.}
\end{equation}
Combining \eqref{intC_0}--\eqref{intC_2} and then letting $\delta \searrow 0$ results in \eqref{intC}.
\end{proof}

\begin{remark}
While the final boundary term in \eqref{intC} (involving $\partial_t u$) is expected to vanish in our applications of Theorem \ref{T.preciseC}, it has to be treated especially delicately.
Thus is due to the presence of $\partial_t u$, which counts for two spatial derivatives in the context of parabolic equations, and which makes this the least regular boundary term.
In particular, we will have to take full advantage of the structure of our heat operator in order to ensure that this term is well-defined and finite.
\end{remark}

\section{Boundary Observability} \label{S.obs}

As an application of Theorem \ref{T.preciseC}, we present in this section a boundary observability result for critically singular (backwards) heat equations.
Throughout, we let $\Omega$, $\Ga$, $d_\Ga$, and the constant $d_0$ be as in previous sections.

Before stating our key results, we must first develop the requisite well-posedness theory for our singular heat operators.
For this, we will also have to treat the more general inhomogeneous extension of Problem (O):

\begin{problem}[OI]
Given final data $u_T$ on $\Omega$, and forcing term $F$ on $( 0, T ) \times \Omega$, solve the following final-boundary value problem for $u$,
\begin{align}
\label{heat_ex} ( \partial_t + \Delta_\sigma + X \cdot \nabla + V ) u = F &\quad \text{on $( 0, T ) \times \Omega$,} \\
\notag u ( T ) = u_T &\quad \text{on $\Omega$,} \\
\notag u = 0 &\quad \text{on $( 0, T ) \times \Gamma$,}
\end{align}
where $\sigma \in ( -\frac{3}{4}, 0 )$, and where the lower-order coefficients satisfy $( X, V ) \in \mc{Z}$.
\end{problem}

Our analysis of Problem (OI) is closely connected to the setting studied in \cite{BZuazua} (but only for subcritical $\sigma$).
Since we are dealing with boundary rather than interior observability, here we must deal more carefully with boundary asymptotics.
Moreover, the presence of lower-order terms in \eqref{heat_ex} complicates the analysis.
As a result, we provide abridged proofs of several key results for completeness.

\begin{remark} \label{time_reverse}
We note that all the theory in this section applies to the \emph{forward} heat equation as well, with the final data $u_T$ replaced by initial data $u_0$ at $t = 0$.
Indeed, this can be obtained by applying the time transformation $t \mapsto T - t$.
\end{remark}

For future convenience, we also use Lemma \ref{L.convex} to fix the following:

\begin{setting}[Boundary defining function]
Fix a boundary defining function $y \in C^4 ( \Omega )$, as given in Definition \ref{ep-bdf}.
(The associated constants $\ep$, $\ep'$ are not relevant.)
\end{setting}

The above is mainly for technical simplification, as this allows us to replace $d_\Ga$, which can fail to be differentiable away from $\Ga$, by a smoother quantity.

\begin{remark} \label{wp_y}
Note the equation \eqref{heat_ex} can now be rewritten as
\begin{equation}
\label{heat_ex_y} \partial_t u + y^{-\kappa} \nabla \cdot [ y^{2\kappa} \nabla ( y^{-\kappa} u ) ] + X \cdot \nabla u + V_y u = F \text{,}
\end{equation}
where the modified potential $V_y$ is given by
\begin{equation} \label{Vy}
V_y = V - \ka y^{-1} \Delta y \,\, \phi + \ka ( 1 - \ka ) ( | \nabla y |^2 y^{-2} - d_\Ga^{-2} ) \, \phi \text{.}
\end{equation}
Note in particular that $( X, V_y ) \in \mc{Z}$.
In the upcoming analysis, it will often be more convenient to express $\Delta_\sigma$ in terms of ``$y$-twisted" derivatives, $y^\ka \nabla y^{-\ka}$ and $y^{-\ka} \nabla y^\ka$.
\end{remark}

\subsection{Elliptic and Semigroup Theory}

The first task is to establish the elliptic and semigroup properties for the singular operator $\Delta_\sigma + X \cdot \nabla + V$.

The following Hardy inequality will play a crucial role in our analysis:

\begin{proposition} \label{prop_hardy_main}
The following inequality holds for any $\phi \in H^1_0 ( \Omega )$:
\begin{equation}
\label{hardy_main} \tfrac{1}{4} \int_\Omega d_\Ga^{-2} \phi^2 \leq \int_\Omega | \nabla \phi |^2 \text{.}
\end{equation}
\end{proposition}

\begin{remark}
See \cite{BM, MMP} for details on Proposition \ref{prop_hardy_main}.
We mention that the explicit constant $\frac{1}{4}$ in \eqref{hardy_main} is only valid when $\Ga$ is convex; for more general $\Om$ and $\Ga$, one still has \eqref{hardy_main}, but with $\frac{1}{4}$ replaced by a possibly smaller positive constant.
\end{remark}

\begin{corollary} \label{prop_hardy_norm}
The following holds for any $\sigma \in ( -\frac{3}{4}, 0 )$ and $\phi \in H^1_0 ( \Omega )$,
\begin{equation}
\label{hardy_norm_equiv} \| \phi \|_{ H^1 ( \Omega ) } \simeq \| y^\kappa \nabla ( y^{-\kappa} \phi ) \|_{ L^2 ( \Omega ) } + \| \phi \|_{ L^2 ( \Omega ) } \text{,}
\end{equation}
where the constants depend on $\Omega$ and $\sigma$.
\end{corollary}

\begin{proof}
Half of \eqref{hardy_norm_equiv} is an immediate consequence of \eqref{hardy_main}:
\[
\| y^\ka \nabla ( y^{-\ka} \phi ) \|_{ L^2 ( \Omega ) } \lesssim \| \nabla \phi \|_{ L^2 ( \Omega ) } + \| y^{-1} \phi \|_{ L^2 ( \Omega ) } \lesssim \| \nabla \phi \|_{ L^2 ( \Omega ) } \text{.}
\]
For the reverse inequality, we integrate by parts to obtain, for $\phi \in C^\infty_0 ( \Om )$,
\begin{align*}
\int_\Omega | \nabla \phi |^2 &\leq \int_\Om \phi ( - \Delta \phi - \sigma y^{-2} \, \phi ) \\
&\leq - \int_\Om \phi \, ( y^{-\kappa} \nabla \cdot [ y^{2\kappa} \nabla ( y^{-\kappa} \phi ) ] ) + [ \| \nabla \phi \|_{ L^2 ( \Omega ) } + \| \phi \|_{ L^2 ( \Omega ) } ] \| \phi \|_{ L^2 ( \Omega ) } \\
&\leq \int_\Om | y^\ka \nabla ( y^{-\ka} \phi ) |^2 + [ \| \nabla \phi \|_{ L^2 ( \Omega ) } + \| \phi \|_{ L^2 ( \Omega ) } ] \| \phi \|_{ L^2 ( \Omega ) } \text{.}
\end{align*}
The result now follows from the above via approximation.
\end{proof}

\begin{remark}
One can in fact show, using \eqref{hardy_main}, that \eqref{hardy_norm_equiv} holds for all $\sigma < \frac{1}{4}$.
\end{remark}

Next, we show that $\De_\sigma + X \cdot \nabla + V$ generates an appropriate semigroup, from which one can derive well-posedness properties for Problem (OI):

\begin{proposition} \label{elliptic}
Fix $\sigma \in ( -\frac{3}{4}, 0 )$ and $( X, V ) \in \mc{Z}$, and consider the operator
\begin{equation}
\label{A_sigma} A_\sigma := \Delta_\sigma + X \cdot \nabla + V \text{,}
\end{equation}
which we view as an unbounded operator on $L^2 ( \Omega )$,
\[
A_\sigma: \mf{D} ( A_\sigma ) \rightarrow L^2 ( \Omega ) \text{,} \qquad \mf{D} ( A_\sigma ) := \{ \phi \in H^1_0 ( \Omega ) \mid A_\sigma \phi \in L^2 ( \Omega ) \} \text{.}
\]
Then, there exists $\gamma \geq 0$ such that:
\begin{itemize}
\item $\lambda I - A_\sigma$ is invertible for any $\lambda > \gamma$, and
\begin{equation}
\label{resolvent} \| ( \lambda I - A_\sigma )^{-1} f \|_{ L^2 ( \Omega ) } \leq ( \lambda - \gamma )^{-1} \| f \|_{ L^2 ( \Omega ) } \text{,} \qquad f \in L^2 ( \Omega ) \text{.}
\end{equation}

\item $-A_\sigma$ generates a $\gamma$-contractive semigroup $t \mapsto e^{-t A_\sigma}$ on $L^2 ( \Omega )$, that is,
\begin{equation}
\label{semigroup} \| e^{-t A_\sigma } \phi \|_{ L^2 ( \Omega ) } \leq e^{ \gamma t } \| \phi \|_{ L^2 ( \Omega ) } \text{,} \qquad t > 0 \text{,} \quad \phi \in L^2 ( \Omega ) \text{.}
\end{equation}
\end{itemize}
Furthermore, if $\phi \in \mf{D} ( A_\sigma )$, then $\phi \in H^2_{\text{loc}} ( \Omega )$, and
\begin{equation}
\label{elliptic_H2} \| y^{-\kappa} \nabla [ y^{2\kappa} \nabla ( y^{-\kappa} \phi ) ] \|_{ L^2 ( \Omega ) } + \| \nabla \phi \|_{ L^2 ( \Omega ) } \lesssim \| A_\sigma \phi \|_{ L^2 ( \Omega ) } + \| \phi \|_{ L^2 ( \Omega ) } \text{,}
\end{equation}
with the constant of the inequality depending only on $\Omega$, $\sigma$, $X$, $V$.
\end{proposition}

\begin{proof}[Proof sketch.]
First, note by the computations in Remark \ref{wp_y}, we have
\begin{equation}
\label{A_sigma_ex} A_\sigma = y^{-\kappa} \nabla \cdot ( y^{2\kappa} \nabla y^{-\kappa} ) + X \cdot \nabla + V_y \text{.}
\end{equation}

We begin with the resolvent estimate \eqref{resolvent}.
Note $-A_\sigma$ can be associated with the bilinear form $\mc{B}_\sigma: H^1_0 ( \Omega ) \times H^1_0 ( \Omega ) \rightarrow \R$, given by
\begin{equation}
\label{elliptic_bilinear} \mc{B}_\sigma ( \phi, \psi ) := \int_\Omega [ y^\ka \nabla ( y^{-\ka} \phi ) \cdot y^\ka \nabla ( y^{-\ka} \psi ) - ( X \cdot \nabla \phi ) \, \psi - V_y \phi \psi ] \text{.}
\end{equation}
By Definition \ref{lower_order_def} and \eqref{hardy_norm_equiv}, there exist $c > 0$ and $\gamma \geq 0$ such that
\begin{equation}
\label{elliptic_energy} \mc{B}_\sigma ( \phi, \phi ) \geq c \| \phi \|_{ H^1 ( \Omega ) }^2 - \gamma \| \phi \|_{ L^2 ( \Omega ) }^2 \text{.}
\end{equation}
In particular, when $\lambda > \gamma$, the Lax-Milgram theorem and \eqref{elliptic_energy} imply that for any $f \in L^2 ( \Omega )$, there exists a unique $\phi \in H^1_0 ( \Omega )$ such that 
\begin{equation}
\label{elliptic_weak} \lambda \int_\Omega \phi \psi + \mc{B}_\sigma ( \phi, \psi ) = \int_\Omega f \psi \text{,} \qquad \psi \in H^1_0 ( \Omega ) \text{.}
\end{equation}
Applying an integration by parts to \eqref{elliptic_weak}, we see that $f = ( \lambda I - A_\sigma ) \phi$ (at least in a weak sense).
Moreover, setting $\psi := \phi$ in \eqref{elliptic_weak} and recalling \eqref{elliptic_energy} yields
\[
( \lambda - \gamma ) \| \phi \|_{ L^2 ( \Omega ) }^2 \leq \| f \|_{ L^2 ( \Omega ) } \| \phi \|_{ L^2 ( \Omega ) } \text{,}
\]
from which the resolvent inequality \eqref{resolvent} immediately follows.

The next step is to obtain the $H^2$-estimate \eqref{elliptic_H2}.
The $H^1$-bound
\begin{equation}
\label{elliptic_H1} \| \nabla \phi \|_{ L^2 ( \Omega ) } \lesssim \| A_\sigma \phi \|_{ L^2 ( \Omega ) } + \| \phi \|_{ L^2 ( \Omega ) }
\end{equation}
is a consequence of \eqref{elliptic_bilinear}, \eqref{elliptic_energy}, and an integration by parts.
Moreover, interior regularity for $A_\sigma$ follows from standard elliptic theory (see \cite[Section 6.3]{Evans}), since all the coefficients of $A_\sigma$ are bounded on any compact subset of $\Omega$.
In particular, $\phi \in \mf{D} ( A_\sigma )$ implies $\phi \in H^2_{loc} ( \Omega )$, and hence it suffices to bound $y^{-\kappa} \nabla [ y^{2\kappa} \nabla ( y^{-\kappa} \phi ) ]$ in \eqref{elliptic_H2} while assuming that $\phi$ is supported sufficiently near $\Gamma$.

Let $\snabla$ and $\sDe$ denote the gradient and Laplacian on the level sets of $y$, respectively.
The informal idea is then to integrate by parts the identity
\begin{equation}
\label{elliptic_H2_pre} \int_\Omega A_\sigma \phi \sDe \phi = \int_\Omega \{ y^{-\kappa} \nabla \cdot [ y^{2\kappa} \nabla ( y^{-\kappa} \phi ) ] + X \cdot \nabla \phi + V_y \phi \} \, \sDe \phi \text{.}
\end{equation}
In particular, estimating lower-order terms using Definition \ref{lower_order_def} and \eqref{hardy_main}, and noting that $\snabla \phi$ and $\snabla^2 \phi$ have zero trace on $\Gamma$, we obtain the estimate
\begin{align*}
&\| y^{-\kappa} \snabla [ y^{2\kappa} \nabla ( y^{-\kappa} \phi ) ] \|_{ L^2 ( \Omega ) }^2 \\
&\quad \lesssim \| A_\sigma \phi \|_{ L^2 ( \Omega ) } \| \sDe \phi \|_{ L^2 ( \Omega ) } + \| \phi \|_{ L^2 ( \Omega ) } \| y^{-\kappa} \snabla [ y^{2\kappa} \nabla ( y^{-\kappa} \phi ) ] \|_{ L^2 ( \Omega ) } \\
&\quad\qquad + \| \phi \|_{ H^1 ( \Omega ) }^2 + \| y^{-1} \phi \|_{ L^2 ( \Omega ) } \| \sDe \phi \|_{ L^2 ( \Omega ) } \\
&\quad \lesssim [ \| A_\sigma \phi \|_{ L^2 ( \Omega ) } + \| \phi \|_{ H^1 ( \Omega ) } ] \| y^{-\kappa} \snabla [ y^{2\kappa} \nabla ( y^{-\kappa} \phi ) ] \|_{ L^2 ( \Omega ) } + \| \phi \|_{ H^1 ( \Omega ) }^2 \text{.}
\end{align*}
(Formally, there is not enough regularity to carry out the above manipulations, and one must approximate, e.g., by replacing $\sDe \phi$ in \eqref{elliptic_H2_pre} with appropriate difference quotients; see \cite[Section 6.3]{Evans}.)
The above then implies
\begin{equation}
\label{elliptic_H2_1} \| y^{-\kappa} \snabla [ y^{2\kappa} \nabla ( y^{-\kappa} \phi ) ] \|_{ L^2 ( \Omega ) }^2 \lesssim \| A_\sigma \phi \|_{ L^2 ( \Omega ) }^2 + \| \phi \|_{ H^1 ( \Omega ) }^2 \text{.}
\end{equation}

In addition, for normal derivatives, we bound, using \eqref{hardy_main}, \eqref{A_sigma}, and \eqref{elliptic_H2_1},
\begin{align}
\label{elliptic_H2_2} \| y^{-\kappa} D_y [ y^{2\kappa} D_y ( y^{-\kappa} \phi ) ] \|_{ L^2 ( \Omega ) } &\lesssim \| A_\sigma \phi \|_{ L^2 ( \Omega ) } + \| \sDe \phi \|_{ L^2 ( \Omega ) } + \| \phi \|_{ H^1 ( \Omega ) } \\
\notag &\lesssim \| A_\sigma \phi \|_{ L^2 ( \Omega ) } + \| \phi \|_{ H^1 ( \Omega ) } \text{.}
\end{align}
The desired estimate \eqref{elliptic_H2} now follows by combining \eqref{elliptic_H1}, \eqref{elliptic_H2_1}, and \eqref{elliptic_H2_2}.

It remains to prove the semigroup property for $-A_\sigma$.
By the Hille--Yosida theorem (see, e.g., the discussions in \cite[Section 7.4]{Evans}), this is a consequence of \eqref{resolvent}, provided we show that $A_\sigma$ is closed and densely defined.
The latter property holds, since $\mf{D} ( A_\sigma )$ contains $C^\infty_0 ( \Omega )$ and hence is dense in $L^2 ( \Omega )$.

Finally, to see that $A_\sigma$ is closed, consider a sequence $( \phi_k )$ in $\mf{D} ( A_\sigma )$ such that
\begin{equation}
\label{closed_cauchy} \lim_{ k \rightarrow \infty } \phi_k = \phi \text{,} \qquad \lim_{ k \rightarrow \infty } A_\sigma \phi_k = \psi \text{,}
\end{equation}
with both limits in $L^2 ( \Omega )$.
Then, all the $\phi_k$'s lie in $H^2_{loc} ( \Omega )$, and \eqref{elliptic_H2} yields that
\begin{align*}
&\| y^{-\kappa} \nabla [ y^{2\kappa} \nabla ( y^{-\kappa} ( \phi_k - \phi_l ) ) ] \|_{ L^2 ( \Omega ) } + \| \nabla ( \phi_k - \phi_l ) \|_{ L^2 ( \Omega ) } \\
&\quad \lesssim \| A_\sigma \phi_k - A_\sigma \phi_l \|_{ L^2 ( \Omega ) } + \| \phi_k - \phi_l \|_{ L^2 ( \Omega ) } \text{,}
\end{align*}
for any $k, l \in \N$.
Since the right-hand side of the above goes to zero as $k, l \rightarrow \infty$ by \eqref{closed_cauchy}, then $( \phi_k )$ is a Cauchy sequence in a weighted $H^2$-space, and
\[
\lim_{ k \rightarrow \infty } \nabla \phi_k = \nabla \phi \text{,} \qquad \lim_{ k \rightarrow \infty } y^{-\kappa} \nabla [ y^{2\kappa} \nabla ( y^{-\kappa} \phi_k ) ] = y^{-\kappa} \nabla [ y^{2\kappa} \nabla ( y^{-\kappa} \phi ) ] \text{.}
\]
The above then implies $\psi = A_\sigma \phi$, and hence $A_\sigma$ is indeed closed.
\end{proof}

\begin{remark}
Hardy's inequality ensures the usual Sobolev space $H^1_0 ( \Omega )$ suffices for working at the level of first derivatives.
However, the situation changes for second derivatives, as the left-hand side of \eqref{elliptic_H2} is no longer comparable to the $H^2$-norm.
\end{remark}

\subsection{Strict Solutions}

Following the discussions in \cite{Biccari, BZuazua}, we now define two notions of solutions of \eqref{heat_ex}, and we state the corresponding well-posedness results:

\begin{definition} \label{mild_soln}
Given $u_T \in L^2 ( \Omega )$ and $F \in L^2 ( (0, T) \times \Omega )$, we call
\[
u \in C^0 ( [ 0, T ]; L^2 ( \Omega ) ) \cap L^2 ( ( 0, T ); H^1_0 ( \Omega ) )
\]
a \emph{mild solution} of Problem (OI) iff the following holds:
\begin{equation}
\label{mild_duhamel} u (t) = e^{(T-t) A_\sigma} u_T - \int_t^T e^{(s-t) A_\sigma} F (s) \, ds \text{,} \qquad t \in [ 0, T ] \text{.}
\end{equation}
\end{definition}

\begin{proposition} \label{wp_mild}
Suppose $u_T \in L^2 ( \Omega )$ and $F \in L^2 ( (0, T) \times \Omega )$.
Then, there is a unique mild solution $u$ to Problem (OI).
Furthermore, $u$ satisfies the estimate
\begin{align}
\label{energy_mild} &\| u \|_{ L^\infty ( [0, T]; L^2 ( \Omega ) ) }^2 + \| y^\ka \nabla ( y^{-\ka} u ) \|_{ L^2 ( ( 0, T ) \times \Omega ) }^2 \\
\notag &\quad \lesssim \| u_T \|_{ L^2 ( \Omega ) }^2 + \| F \|_{ L^2 ( ( 0, T ) \times \Omega ) }^2 \text{,}
\end{align}
with the constant of the inequality depending only on $\Omega$, $\sigma$, $X$, $V$.
\end{proposition}

\begin{proof}[Proof sketch.]
Both existence and uniqueness are immediate from \eqref{mild_duhamel}.
For \eqref{energy_mild}, we only consider when $u_T \in \mf{D} ( A_\sigma )$ (so that $u (t) \in \mf{D} ( A_\sigma )$ and $\partial_t u (t) \in L^2 ( \Omega )$ for every $t \in [ 0, T )$); the general case then follows by approximation.

By the fundamental theorem of calculus, \eqref{heat_ex_y}, and integrations by parts,
\begin{align*}
&\| u (T) \|_{ L^2 ( \Omega ) }^2 - \| u (t) \|_{ L^2 ( \Omega ) }^2 \\
&\quad = 2 \int_t^T \int_\Omega u \{ F - y^{-\ka} \nabla \cdot [ y^{2\ka} \nabla ( y^{-\ka} u ) ] - X \cdot \nabla u - V_y u \} \big|_{t=s} ds \\
&\quad = 2 \int_t^T \int_\Omega F u |_{t=s} ds + 2 \int_t^T \int_\Omega | y^\ka \nabla ( y^{-\ka} u ) |^2 \big|_{ t = s } ds \\
&\quad\qquad + \int_t^T \int_\Omega ( \nabla \cdot X - 2 V_y ) u^2 \big|_{t=s} ds \text{,}
\end{align*}
for any $t \in [ 0, T )$.
Rearranging and recalling Definition \ref{lower_order_def}, we obtain that
\begin{align*}
&\| u (t) \|_{ L^2 ( \Omega ) }^2 + 2 \int_t^T \int_\Omega | y^\ka \nabla ( y^{-\ka} u (s) ) |^2 \, ds \\
&\quad \leq \| u_T \|_{ L^2 ( \Omega ) }^2 + \int_t^T [ \| F (s) \|_{ L^2 ( \Omega ) } + \| y^{-1} u (s) \|_{ L^2 ( \Omega ) } ] \| u (s) \|_{ L^2 ( \Omega ) } ds \text{.}
\end{align*}
Applying \eqref{hardy_main}, \eqref{hardy_norm_equiv}, and absorbing then yields,
\begin{align*}
\| u (t) \|_{ L^2 ( \Omega ) }^2 + \| u \|_{ L^2 ( ( 0, T ); H^1 ( \Omega ) ) }^2 &\lesssim \| u_T \|_{ L^2 ( \Omega ) }^2 + \| F \|_{ L^2 ( (0, T) \times \Omega ) }^2 \\
&\qquad + \int_t^T \| u (s) \|_{ L^2 ( \Omega ) }^2 ds \text{,}
\end{align*}
and the result follows from Gronwall's inequality.
\end{proof}

\begin{definition} \label{strict_soln}
Given $u_T \in H^1_0 ( \Omega )$ and $F \in L^2 ( (0, T) \times \Omega )$, we call
\[
u \in C^0 ( [ 0, T ]; H^1_0 ( \Omega ) ) \cap H^1 ( ( 0, T ); L^2 ( \Omega ) ) \cap L^2 ( ( 0, T ); \mf{D} ( A_\sigma ) )
\]
a \emph{strict solution} of Problem (OI) iff:
\begin{itemize}
\item $( \partial_t + \Delta_\sigma + X \cdot \nabla + V ) u = F$ almost everywhere on $( 0, T ) \times \Omega$.

\item $u ( T ) = u_T$ holds as an equality in $H^1_0 ( \Omega )$.
\end{itemize}
\end{definition}

\begin{proposition} \label{wp_strict}
Suppose $u_T \in H^1_0 ( \Omega )$ and $F \in L^2 ( (0, T) \times \Omega )$.
Then, the mild solution $u$ from Proposition \ref{wp_mild} is also the unique strict solution to Problem (OI).
Furthermore, $u$ satisfies the energy inequality
\begin{align}
\label{energy_strict} &\| u \|_{ L^\infty ( [0, T]; H^1 ( \Omega ) ) }^2 + \| y^{-\ka} \nabla [ y^{2\ka} \nabla ( y^{-\ka} u ) ] \|_{ L^2 ( ( 0, T ) \times \Omega ) }^2 \\
\notag &\quad \lesssim \| u_T \|_{ H^1 ( \Omega ) }^2 + \| F \|_{ L^2 ( ( 0, T ) \times \Omega ) }^2 \text{,}
\end{align}
again with the constant depending only on $\Omega$, $\sigma$, $X$, $V$.
\end{proposition}

\begin{proof}[Proof sketch.]
That the mild solution is also the strict solution is immediate.
For \eqref{energy_strict}, we again need only consider $u_T \in \mf{D} ( A_\sigma )$.

By the fundamental theorem of calculus, integrations by parts, and \eqref{heat_ex_y},
\begin{align*}
&\| y^\ka \nabla ( y^{-\ka} u (T) ) \|_{ L^2 ( \Omega ) }^2 - \| y^\ka \nabla ( y^{-\ka} u (t) ) \|_{ L^2 ( \Omega ) }^2 \\
&\quad = - 2 \int_t^T \int_\Om \partial_t u \,\, y^{-\ka} \nabla \cdot [ y^{2\ka} \nabla ( y^{-\ka} u ) ] \big|_{t=s} ds \\
&\quad = 2 \int_t^T \int_\Om ( - F + X \cdot \nabla u + V_y u ) \,\, y^{-\ka} \nabla \cdot [ y^{2\ka} \nabla ( y^{-\ka} u ) ] \big|_{t=s} ds \\
&\quad\qquad + 2 \int_t^T \int_\Om | y^{-\ka} \nabla \cdot [ y^{2\ka} \nabla ( y^{-\ka} u ) ] |^2 \big|_{t=s} ds \text{.}
\end{align*}
Rearranging the above and applying Hardy's inequality then yields
\begin{align*}
&\| u \|_{ L^\infty ( [0, T]; H^1 ( \Omega ) ) }^2 + \| y^{-\ka} \nabla \cdot [ y^{2\ka} \nabla ( y^{-\ka} u ) ] \|_{ L^2 ( ( 0, T ) \times \Omega ) }^2 \\
&\quad \lesssim \| u_T \|_{ H^1 ( \Omega ) }^2 + \| F \|_{ L^2 ( ( 0, T ) \times \Omega ) }^2 \text{.}
\end{align*}
The desired \eqref{energy_strict} now follows from the above and from \eqref{elliptic_H2}.
\end{proof}

\begin{remark}
While our well-posedness theory only applies when the lower-order coefficients $X$ and $V$ are time-independent, this restriction is not essential.
In fact, one can also treat time-dependent $X$ and $V$ using a Galerkin method approach; see \cite{Warnick}, which develops this theory for critically singular hyperbolic equations.
\end{remark}

\subsection{The Neumann Trace} \label{sec.neumann}

From now on, we will focus mainly be on strict solutions to Problem (OI), which are particularly relevant as this level of regularity is sufficient to define and control the Neumann boundary trace.

\begin{proposition} \label{neumann_trace}
Fix $u_T \in H^1_0 ( \Omega )$ and $F \in L^2 ( (0, T) \times \Omega )$, and let $u$ denote the strict solution to Problem (OI) (with this $u_T$ and $F$).
Then, the Neumann trace $\mc{N}_\sigma u$ is well-defined in $L^2 ( (0, T) \times \Gamma )$ and satisfies the bound
\begin{equation}
\label{hidden_reg} \| \mc{N}_\sigma u \|_{ L^2 ( ( 0, T ) \times \Gamma ) }^2 \lesssim \| u_T \|_{ H^1 ( \Omega ) }^2 + \| F \|_{ L^2 ( ( 0, T ) \times \Omega ) }^2 \text{,}
\end{equation}
where the constant in the above depends on $\Omega, \sigma, X, V$.

Furthermore, the following limit holds in $L^2 ( ( 0, T ) \times \Gamma )$:
\begin{equation}
\label{hidden_dirichlet} \lim_{ d_\Ga \rightarrow 0 } d_\Ga^{-1+\kappa} u = \tfrac{1}{ 1 - 2 \kappa } \mc{N}_\sigma u \text{.}
\end{equation}
\end{proposition}

\begin{proof}
For any $x \in \Ga$ and $0 < y_0 < 2 d_0$, we let $\eta_{ y_0 } (x)$ denote the point on the level set $\{ y = y_0 \}$ that is reached from $x$ along the integral curve of $\nabla y$.
Letting $d S$ be the surface measure on $\Ga$, then for any $0 < y_0' < y_0 < 2 d_0$,
\begin{align*}
&\int_{ ( 0, T ) \times \Ga } \Big[ y^{2\ka} D_y ( y^{-\ka} u ) \big|_{ ( t, \eta_{ y_0 } (x) ) } - y^{2\ka} D_y ( y^{-\ka} u ) \big|_{ ( t, \eta_{ y_0' } (x) ) } \Big]^2 dS (x) \, dt \\
&\quad = \int_{ ( 0, T ) \times \Ga } \bigg( \int_{ y_0' }^{ y_0 } D_y [ y^{2\ka} D_y ( y^{-\ka} u ) ] \big|_{ ( t, \eta_y (x) ) } dy \bigg)^2 dS (x) dt \\
&\quad \leq \int_{ y_0' }^{ y_0 } y^{2\ka} dy \cdot \int_{ ( 0, T ) \times \Ga } \int_{ y_0' }^{ y_0 } \Big| y^{-\ka} D_y [ y^{2\ka} D_y ( y^{-\ka} u ) ] \big|_{ ( t, \eta_y (x) ) } \Big|^2 dy \, dS (x) dt \\
&\quad \lesssim (1+2\ka)^{-1} y_0^{1+2\ka} \cdot \int_0^T \int_\Om | y^{-\ka} \nabla [ y^{2\ka} \nabla ( y^{-\ka} u (s) ) ] |^2 \, ds \text{,}
\end{align*}
where we used that $2 \ka > -1$ and that $y = d_\Ga$ near $\Ga$.
By the inequality \eqref{energy_strict}, the right-hand side of the above vanishes when $y_0 \searrow 0$, and it hence follows that $\mc{N}_\sigma u$ exists as an element of $L^2 ( ( 0, T ) \times \Gamma )$.

Next,	let $\chi: \R \rightarrow [ 0, 1 ]$ be a cutoff function satisfying
\[
\chi (s) = \begin{cases} 1 & s < d_0 \text{,} \\ 0 & s > \tfrac{3 d_0}{2} \text{.} \end{cases}
\]
Then, a similar computation as before, again using that $2 \ka > -1$, yields 
\begin{align*}
\int_{ ( 0, T ) \times \Ga } ( \mc{N}_\sigma u)^2 &= \int_{ ( 0, T ) \times \Ga } \bigg( \int_{0}^{2 d_0} D_y [ \chi (y) \, y^{2\ka} D_y(y^{-\ka } u) ] \big|_{ ( t, \eta_y (x) ) } dy \bigg)^2 dS(x) dt \\
&\lesssim d_0^{1+2\ka} \int_{ ( 0, T ) \times \Ga } \int_0^{2 d_0} \Big| y^{-\ka} D_y [ y^{2\ka} D_y (y^{-\ka} u) ] \big|_{ ( t, \eta_y (x) ) } \Big|^2 dy \, dS (x) dt \\
&\qquad + d_0^{1+2\ka} \int_{ ( 0, T ) \times \Ga } \int_0^{2 d_0} \Big| y^\ka D_y (y^{-\ka} u) \big|_{ ( t, \eta_y (x) ) } \Big|^2 dy \, dS (x) dt \\
&\lesssim \int_0^T \int_\Om | y^{-\ka} \nabla [ y^{2\ka} \nabla ( y^{-\ka} u (s) ) ] |^2 \, ds + \| u \|_{ L^2 ( ( 0, T ); H^1 ( \Omega ) ) }^2 \text{,}
\end{align*}
where we also used \eqref{hardy_norm_equiv}.
The bound \eqref{hidden_reg} follows from \eqref{energy_strict} and the above.

Next, for \eqref{hidden_dirichlet}, we first note, for any $0 < y_0 < 2 d_0$, that
\begin{align*}
&\int_{ ( 0, T ) \times \Ga } \Big( y^{\ka-1} u \big|_{ ( t, \eta_{ y_0 } (x) ) } - \tfrac{1}{1-2\ka} \mc{N}_\sigma u \big|_{ ( t, x ) } \Big)^2 dS (x) dt \\
&\quad = \int_{ ( 0, T ) \times \Ga } \bigg[ y_0^{2\ka-1} \int_0^{y_0} D_y (y^{-\ka} u) \big|_{ ( t, \eta_y (x) ) } dy - \tfrac{1}{1-2\ka} \mc{N}_\sigma u \big|_{ ( t, x ) } \bigg]^2 dS(x) dt \\
&\quad = \int_{ ( 0, T ) \times \Ga } \bigg( y_0^{2\ka-1}\int_{0}^{y_0}y^{-2\ka} \big[ y^{2\ka} D_y (y^{-\ka} u) \big|_{ ( t, \eta_y (x) ) } - \mc{N}_\sigma u \big|_{ ( t, x ) } \big] dy \bigg)^2 dS(x) dt \text{,}
\end{align*}
where we used the boundary condition $u = 0$ on $\Ga$ from \eqref{heat_ex}, and that $\ka < 0$.
We next employ Minkowski's inequality on the above to derive
\begin{align*}
&\bigg[ \int_{ ( 0, T ) \times \Ga } \Big( y^{\ka-1} u \big|_{ ( t, \eta_y (x) ) } - \tfrac{1}{1-2\ka} \mc{N}_\sigma u \big|_{ ( t, x ) } \Big)^2 dS (x) dt \bigg]^\frac{1}{2} \\
&\quad \lesssim y_0^{2\ka-1} \int_0^{y_0} y^{-2\ka} \bigg[ \int_{ ( 0, T ) \times \Ga } \Big( y^{2\ka} D_y (y^{-\ka} u) \big|_{ ( t, \eta_y (x) ) } - \mc{N}_\sigma u \big|_{ ( t, x ) } \Big)^2 dS(x) dt \bigg]^\frac{1}{2} d y \\
&\quad \lesssim \sup_{ 0 < y < y_0 } \bigg[ \int_{ ( 0, T ) \times \Ga } \Big( y^{2\ka} D_y (y^{-\ka} u) \big|_{ ( t, \eta_y (x) ) } - \mc{N}_\sigma u \big|_{ ( t, x ) } \Big)^2 dS(x) dt \bigg]^\frac{1}{2} \text{.}
\end{align*}
Since $\mc{N}_\sigma u \in L^2 ( ( 0, T ) \times \Ga )$, its definition \eqref{bdry_data} implies the above converges to $0$ as $y_0 \searrow 0$.
This immediately implies the desired limit \eqref{hidden_dirichlet}.
\end{proof}

\begin{remark}
With some modification, one can extend the preceding well-posedness theory (Propositions \ref{wp_mild} and \ref{wp_strict}) and Proposition \ref{neumann_trace} to $0 \leq \sigma < \frac{1}{4}$.
\end{remark}

Next, we prove a technical result, roughly stating that the least regular boundary term in the Carleman estimate \eqref{intC} indeed vanishes in our present setting:

\begin{proposition} \label{stupid_boundary}
Let $u_T \in H^1_0 ( \Omega )$, and let $u$ denote the strict solution to Problem (O) (that is, Problem (OI) without forcing term $F \equiv 0$).
Then, $u$ satisfies
\begin{align}
\label{stupid_limit} \lim_{ \delta \searrow 0 } \int_{ (0,T) \times \{ y = \delta \} } e^{-2 \lambda \mc{F}} \, \partial_t ( y^{-\ka} u ) \, y^{2\ka} D_y ( y^{-\ka} u ) &= 0 \text{,} \\
\notag \lim_{ \delta \searrow 0 } \int_{ (0,T) \times \{ y = \delta \} } e^{-2 \lambda \mc{F}} \, \partial_t ( y^{-\ka} u ) \, y^{-1+\ka} u &= 0 \text{,}
\end{align}
for any $\lambda > 0$, where $\mc{F}$ denotes the weight
\begin{equation}
\label{stupid_F} \mc{F} ( t, x ) := \tfrac{1}{ t (T - t) } [ y (x)^{ 1 + 2 \ka } + \beta ] \text{,} \qquad \beta > 0 \text{.}
\end{equation}
\end{proposition}

\begin{proof}[Proof sketch.]
Define the bilinear maps $\mc{B}_1, \mc{B}_2: H^1_0 ( \Omega ) \rightarrow \R$ by
\begin{align}
\label{stupid_0} \mc{B}_1 ( u_T ) &:= \lim_{ \delta \searrow 0 } \int_{ ( 0, T ) \times \{ y = \delta \} } e^{-2 \lambda \mc{F}} \, \partial_t ( y^{-\kappa} u ) \, y^{2\kappa} D_y ( y^{-\kappa} u ) \text{,} \\
\notag \mc{B}_2 ( u_T ) &:= \lim_{ \delta \searrow 0 } \int_{ ( 0, T ) \times \{ y = \delta \} } e^{-2 \lambda \mc{F}} \, \partial_t ( y^{-\kappa} u ) \, y^{-1+\kappa} u \text{,}
\end{align}
where $u$ is the strict solution to Problem (O) with final data $u_T$.
It then suffices to show that both $\mc{B}_1$ and $\mc{B}_2$ are everywhere vanishing.

The main step is to show that both $\mc{B}_1$ and $\mc{B}_2$ are well-defined and finite.
The process is similar to the proof of Proposition \ref{neumann_trace}, except we need more care with regularity.
Somewhat informally, we use the divergence theorem to bound $\mc{B}_1$ by
\begin{align}
\label{stupid_10} | \mc{B}_1 ( u_T ) | &= \bigg| \lim_{ \delta \searrow 0 } \int_{ ( 0, T ) \times \{ y > \delta \} } \nabla \cdot [ e^{-2 \lambda \mc{F}} \, \partial_t ( y^{-\kappa} u ) \, y^{2\kappa} \nabla ( y^{-\kappa} u ) ] \bigg| \\
\notag &\leq \tfrac{1}{2} \limsup_{ \delta \searrow 0 } \bigg| \int_{ ( 0, T ) \times \{ y > \delta \} } e^{-2 \lambda \mc{F}} \partial_t [ | y^\ka \nabla ( y^{-\ka} u ) |^2 ] \bigg| \\
\notag &\qquad + C \int_{ ( 0, T ) \times \Om } | \partial_t u | \big| y^{-\ka} \nabla \cdot [ y^{2\ka} \nabla ( y^{-\ka} u ) ] \big| \\
\notag &\qquad + C \int_{ ( 0, T ) \times \Om } | \partial_t u | \, y^\ka | y^{2\ka} D_y ( y^{-\ka} u ) | \\
\notag &\lesssim \sup_{ t \in [ 0, T ] } \int_\Om | y^\ka \nabla ( y^{-\ka} u (t) ) |^2 + \int_{ ( 0, T ) \times \Om } y^{2\ka} | y^{2\ka} D_y ( y^{-\ka} u ) |^2 \\
\notag &\qquad + \int_{ ( 0, T ) \times \Om } \big( | \partial_t u |^2 + \big| y^{-\ka} \nabla \cdot [ y^{2\ka} \nabla ( y^{-\ka} u ) ] \big|^2 \big) \\
\notag &= I_{1, 1} + I_{1, 2} + I_{1, 3} \text{,}
\end{align}
where all constants here and below can depend on $\Om$, $\sigma$, $X$, $V$, as well as $\beta, \lambda$, and where we also noted that $\lambda^3 t^{-3} (T - t)^{-3} e^{-2 \la \mc{F}}$ is bounded.

For $I_{1, 1}$, we apply Corollary \ref{prop_hardy_norm} and Proposition \ref{wp_strict}, which yield
\begin{equation}
\label{stupid_11} I_{1, 1} \lesssim \| u_T \|_{ H^1 ( \Omega ) }^2 \text{.}
\end{equation}
For $I_{ 1, 3 }$, we recall the heat equation \eqref{heat_obs}, Proposition \ref{prop_hardy_main}, and \eqref{energy_strict} to obtain
\begin{equation}
\label{stupid_12} I_{1, 3} \lesssim \| u_T \|_{ H^1 ( \Omega ) }^2 \text{.}
\end{equation}
For $I_{ 1, 2 }$, we integrate the pointwise Hardy inequality of Lemma \ref{pointHardy}, with parameters $q := 1 + \ka$ and $v := y^{2\ka} D_y ( y^{-\ka} u )$, and we recall Propositions \ref{wp_strict} and \ref{neumann_trace}:
\begin{align*}
I_{1, 2} &\lesssim \int_{ ( 0, T ) \times \Om } \big( y^{2+2\ka} | D_y [ y^{2\ka} D_y ( y^{-\ka} u ) ] |^2 + y^{1+2\ka} | y^{2\ka} D_y ( y^{-\ka} u ) |^2 \big) \\
&\qquad + \lim_{ \delta \searrow 0 } \int_{ ( 0, T ) \times \{ y = \delta \} } y^{1+2\ka} | y^{2\ka} D_y ( y^{-\ka} u ) |^2 \\
&\lesssim \| u_T \|_{ H^1 ( \Omega ) }^2 + \int_{ ( 0, T ) \times \Om } y^{1+2\ka} | y^{2\ka} D_y ( y^{-\ka} u ) |^2 \text{.}
\end{align*}
(Formally, the region near the critical point of $y$, where $| \nabla y | = 0$, can be trivially treated.)
Integrating Lemma \ref{pointHardy} again, now with $q := \frac{3}{2} + \ka$, yields
\begin{equation}
\label{stupid_13} I_{1, 2} \lesssim \| u_T \|_{ H^1 ( \Omega ) }^2 \text{.}
\end{equation}
Combining \eqref{stupid_10}--\eqref{stupid_13}, we see that
\[
\mc{B}_1 ( u_T ) \lesssim \| u_T \|_{ H^1 ( \Omega ) }^2 \text{.}
\]
Also, applying the above to differences of final data, we see that $\mc{B}_1$ is continuous.

Next, for $\mc{B}_2$, we have
\begin{align*}
| \mc{B}_2 ( u_T ) | &= \tfrac{1}{2} \left| \int_{ ( 0, T ) \times \Gamma } e^{-2 \lambda \mc{F}} \partial_t ( y^{-1} u^2 ) \right| \\
&\lesssim \sup_{ t \in [ 0, T ] } \lim_{ \delta \searrow 0 } \int_{ \{ y = \delta \} } y^{-1} | u (t) |^2 + \lim_{ \delta \searrow 0 } \int_{ ( 0, T ) \times \{ y = \delta \} } y^{2\ka} u^2 \text{.}
\end{align*}
The last term in the above vanishes by Proposition \ref{neumann_trace}.
For the remaining term, we again integrate Lemma 2.6, with $q = 0$, which yields
\begin{align*}
| \mc{B}_2 ( u_T ) | &\lesssim \sup_{ t \in [ 0, T ] } \int_\Om ( | D_y u |^2 + y^{-1} u^2 ) \\
&\lesssim \| u_T \|_{ H^1 ( \Omega ) }^2 \text{,}
\end{align*}
where in the last step, we also applied Propositions \ref{prop_hardy_main} and \ref{wp_strict}.
Like for $\mc{B}_1$, the above suffices to imply the finiteness and continuity of $\mc{B}_2$.

(Formally, to rigorously show $\mc{B}_1 ( u_T )$, $\mc{B}_2 ( u_T )$ are well-defined, we would need, as in the proof of Proposition 3.14, to estimate differences of the associated integrals over $( 0, T ) \times \{ y = \delta_1 \}$ and $( 0, T ) \times \{ y = \delta_2 \}$, with $\delta_1, \delta_2 \searrow 0$.
However, we skip this step here, as the details of this are analogous to the above.)

Finally, by continuity, it suffices to show $\mc{B}_1$ and $\mc{B}_2$ vanish on a dense subspace of $H^1_0 ( \Omega )$.
For this, we consider the domain $\mf{D} ( A_\sigma )$ from Proposition \ref{elliptic}.
Observe in particular that if $u_T \in \mf{D} ( A_\sigma )$, then the relations (see \eqref{mild_duhamel})
\[
u (t) = e^{ (T-t) A_\sigma } u_T \text{,} \qquad \partial_t u (t) = e^{ (T-t) A_\sigma } ( -A_\sigma u_T )
\]
imply that $\partial_t u$ is a mild solution to Problem (O), with final data $- A_\sigma u_T \in L^2 ( \Omega )$.
Moreover, Proposition \ref{wp_mild} yields $\partial_t u \in L^2 ( ( 0, T ); H^1_0 ( \Omega ) )$, and hence $\mc{D}_\sigma ( \partial_t u ) = 0$ as an element of $L^2 ( ( 0, T ) \times \Gamma )$.
Applying the above to \eqref{stupid_0}, we now have both $\mc{B}_1 ( u_T ) = 0$ and $\mc{B}_2 ( u_T ) = 0$ whenever $u_T \in \mf{D} ( A_\sigma )$, as desired.
\end{proof}

\subsection{Observability}

Lastly, we state the key observability inequality and unique continuation property satisfied by solutions of Problem (O):

\begin{theorem} \label{obs}
Let $u_T \in H^1_0 ( \Omega )$, and let $u$ be the corresponding strict solution to Problem (O).
Then, the following observability estimate holds,
\begin{equation}
\label{obs_ineq} \| u (0) \|_{ H^1 ( \Omega ) }^2 \lesssim \| \mc{N}_\sigma u \|_{ L^2 ( ( 0, T ) \times \Gamma ) }^2 \text{,}
\end{equation}
with the constant of the inequality depending on $\Omega, \sigma, X, V$.

In particular, if $\mc{N}_\sigma u \equiv 0$ on $( 0, T ) \times \Gamma$, then $u \equiv 0$ on $[ 0, T ] \times \Omega$.
\end{theorem}

\begin{proof}
Applying the global Carleman estimate of Theorem \ref{T.preciseC}, with $\sigma := \kappa ( 1 - \kappa )$ from Problem (O) and with $p := \kappa + \frac{1}{2}$, so that
\[
p \in \big( 0, \tfrac{1}{2} \big) \text{,} \qquad p^2 - 2 p + \tfrac{3}{4} = \sigma \text{,}
\]
we see that there exists a boundary defining pair $( y_1, y_2 )$ (again, the values of the associated constants are not important) such that for sufficiently large $\lambda > 0$,
\begin{align}
\label{obs_ineq_1} &\sum_{j=1}^2 \int_{ (0,T) \times \Om } \la \te e^{-2 \la F_j} \big( |\nab u|^2 + y_j^{-2} u^2 \big) \\
\notag &\quad \leq \limsup_{ \delta \searrow 0 } \sum_{ j = 1 }^2 \int_{ (0,T) \times \{ y_j = \delta \} } \la^3 \te^3 e^{-2 \lambda F_j} \big[ y_j^{2\ka} \, ( D_{ y_j } u )^2 + y_j^{-2+2\ka} \, u^2 \big] \\
\notag &\quad\qquad + \limsup_{ \delta \searrow 0 } \sum_{ j = 1 }^2 \bigg| \int_{ (0,T) \times \{ y_j = \delta \} } e^{-2 \lambda F_j} \, \partial_t u D_{ y_j } u \bigg| \\
\notag &\quad\qquad + \sum_{j=1}^2 \int_{ (0,T) \times \Om } e^{-2 \la F_j} ( \partial_t + \Delta_{ \sigma, y_j } u )^2 \text{.}
\end{align}
In the above, $( \theta, F_1, F_2 )$ are defined from $( y_1, y_2 )$ via \eqref{weightmu}, and the left-hand side of \eqref{intC} was further simplified by recalling that $\ka < 0$.
(While \eqref{intC} holds for classical regular solutions, this can be extended to strict solutions via approximation.)

For the first boundary term in \eqref{obs_ineq_1}, we apply Proposition \ref{neumann_trace} to obtain
\begin{align}
\label{obs_ineq_2} &\limsup_{ \delta \searrow 0 } \sum_{ j = 1 }^2 \int_{ (0,T) \times \{ y_j = \delta \} } \la^3 \te^3 e^{-2 \lambda F_j} \big[ y_j^{2\ka} \, ( D_{ y_j } u )^2 + y_j^{-2+2\ka} \, u^2 \big] \\
\notag &\quad \lesssim \limsup_{ \delta \searrow 0 } \sum_{ j = 1 }^2 \int_{ (0,T) \times \{ y_j = \delta \} } \la^3 \te^3 e^{-2 \lambda F_j} \big[ | y_j^{2\ka} D_{ y_j } ( y_j^{-\ka} u ) |^2 + y_j^{-2+2\ka} \, u^2 \big] \\
\notag &\quad \lesssim \int_{ (0,T) \times \Ga } ( \mc{N}_\sigma u )^2 \text{,}
\end{align}
where we also noted in the last step that $\la^3 \te^3 e^{-2\la F_j}$ is bounded.
Moreover, from Proposition \ref{stupid_boundary}, the remaining boundary term in \eqref{obs_ineq_1} vanishes.

Next, from Definition \ref{lower_order_def} and the definition of $y$, we have
\begin{align}
\label{obs_ineq_3} &\sum_{j=1}^2 \int_{ (0,T) \times \Om } e^{-2 \la F_j} ( \partial_t + \Delta_{ \sigma, y_j } u )^2 \\
\notag &\qquad \lesssim \sum_{j=1}^2 \int_{ (0,T) \times \Omega } e^{ -2 \lambda F_j } \big( | \nabla u |^2 + y_j^{-2} u^2 \big) \text{.}
\end{align}
Thus, combining \eqref{obs_ineq_1}--\eqref{obs_ineq_3} and taking $\lambda$ sufficiently large yields
\begin{equation}
\label{obs_ineq_10} \sum_{j=1}^2 \int_{ (0,T) \times \Om } e^{-2 \la F_j} \big( |\nab u|^2 + y_j^{-2} u^2 \big) \lesssim \int_{ (0,T) \times \Ga } ( \mc{N}_\sigma u )^2 \text{.}
\end{equation}
Note \eqref{obs_ineq_10} implies unique continuation---if $\mc{N}_\sigma u \equiv 0$, then $u \equiv 0$ on $[ 0, T ] \times \Omega$.

Finally, applying the Hardy inequality \eqref{hardy_main} to \eqref{obs_ineq_10}, we have
\[
\int_0^T e^{ -c \lambda \theta (t) } \| u (t) \|_{ H^1 ( \Omega ) } \, dt \lesssim \int_{ (0,T) \times \Ga } ( \mc{N}_\sigma u )^2 \text{,} \qquad c > 0 \text{.}
\]
Applying the inequality \eqref{energy_strict} on each interval $( 0, t )$ (with $F \equiv 0$) in the left-hand side of the above, we can estimate the $H^1$-norm of $u (t)$ from below by the $H^1$-norm of $u (0)$.
Since $e^{-c \lambda \theta}$ is clearly integrable on $( 0, T )$, then \eqref{obs_ineq} follows.
\end{proof}

\section{Boundary Controllability} \label{S.control}

In this section, we apply the Neumann regularity (Proposition \ref{neumann_trace}) and boundary observability (Theorem \ref{obs}) for the backward heat equation to prove our main boundary controllability result for the forward heat equation.
In particular, here we are primarily concerned with the setting of Problem (C):
\begin{align*}
( -\partial_t + \Delta_\sigma + Y \cdot \nabla + W ) v = 0 &\quad \text{on $( 0, T ) \times \Omega$,} \\
v ( 0 ) = v_0 &\quad \text{on $\Omega$,} \\
\mc{D}_\sigma v = f &\quad \text{on $( 0, T ) \times \Gamma$.}
\end{align*}
As usual, we adopt the same setting as described in previous sections.

\subsection{Regular Solutions}

The first step is to briefly discuss how solutions of Problem (C) with nonzero Dirichlet data are constructed for sufficiently regular data.

\begin{proposition} \label{reg_dir}
Given $v_0 \in H^1_0 ( \Om )$ and $f \in C^\infty_0 ( ( 0, T ) \times \Ga )$, there exists
\[
v \in C^0 ( [ 0, T ]; H^1_{\text{loc}} ( \Om ) ) \cap H^1 ( ( 0, T ) \times L^2 ( \Om ) ) \cap L^2 ( ( 0, T ); H^2_{\text{loc}} ( \Om ) )
\]
that solves Problem (C) in the following sense:
\begin{itemize}
\item $( -\partial_t + \Delta_\sigma + Y \cdot \nabla + W ) v = 0$ almost everywhere on $( 0, T ) \times \Omega$.

\item $v ( 0 ) = v_0$ holds as an equality in $H^1_{ \text{loc} } ( \Omega )$.

\item $\mc{D}_\sigma v = f$ holds in the trace sense in $C^0 ( [ 0, T ]; L^2 ( \Gamma ) )$.
\end{itemize}
Moreover, if $u_T \in H^1_0 ( \Om )$ and $F \in L^2 ( ( 0, T ) \times \Om )$, and if $u$ is the corresponding strict solution of Problem (OI), with lower-order coefficients given by
\begin{equation}
\label{coeff_duality} ( X, V ) := ( -Y, \, W - \nabla \cdot Y ) \in \mc{Z} \text{,}
\end{equation}
then the following identity holds:
\begin{equation}
\label{duality} \int_{ ( 0, T ) \times \Omega } F v = \int_\Omega u_T \, v (T) - \int_\Omega u (0) \, v_0 + \int_{ ( 0, T ) \times \Gamma } \mc{N}_\sigma u \, f \text{.}
\end{equation}
\end{proposition}

\begin{proof}[Proof sketch.]
For convenience, we adopt the shorthand
\begin{equation}
\label{B_sigma} B_\sigma := \Delta_\sigma + Y \cdot \nabla + W \text{.}
\end{equation}
First, we construct a suitable extension $v_f: ( 0, T ) \times \Omega \rightarrow \R$ of $f$:
\begin{itemize}
\item We extend $f$ to a sufficiently small neighborhood $( 0, T ) \times U_\Ga$ of $( 0, T ) \times \Ga$ by defining it to be constant along the integral curves of $\nabla y$ at each time.
Calling this extension $f_\Ga$, we then define, on $( 0, T ) \times U_\Ga$, the function
\begin{equation}
\label{vf} v_f := y^\ka f_\Ga - \tfrac{1}{ 2 \ka } y^{1+\ka} ( \ka Y \cdot \nabla y + y W_y ) f_\Ga \in C^2 ( ( 0, T ) \times U_\Ga ) \text{.}
\end{equation}

\item $v_f$ is then extended arbitrarily to all of $( 0, T ) \times \Om$ as a $C^2$-function.
\end{itemize}
Furthermore, observe that since $f \in C^\infty_0 ( ( 0, T ) \times \Ga )$, we can also arrange such that $v_f$ smoothly extends to $t = 0$ by the condition $v_f |_{ t = 0 } \equiv 0$.

The key observation is that $( -\partial_t + B_\sigma ) v_f$ lies in $L^2 ( ( 0, T ) \times \Om )$.
To confirm this, we need only check that \eqref{vf} has this property on $( 0, T ) \times U_\Ga$, on which we can assume $y = d_\Ga$.
For the first term on the right-hand side of \eqref{vf}, we have
\begin{align*}
( -\partial_t + B_\sigma ) ( y^\ka f_\Ga ) &= y^{-\ka} \nabla \cdot ( y^{2\ka} \nabla f_\Ga ) + Y \cdot \nabla ( y^\ka f_\Ga ) + W_y \, y^\ka f_\Ga + O ( y^\ka ) \\
&= 2 \kappa y^{\ka-1} \nabla y \cdot \nabla f_\Ga + y^{\ka-1} ( \ka Y \cdot \nabla y + y W_y ) f_\Ga + O ( y^\ka ) \text{,}
\end{align*}
since $f_\Ga$ and its derivatives are bounded up to $\Ga$ by definition.
As $\nabla y \cdot \nabla f_\Ga$ vanishes (again by the definition of $f_\Ga$), we hence obtain
\begin{equation}
\label{duality_1} ( -\partial_t + B_\sigma ) ( y^\ka f_\Ga ) = y^{\ka-1} ( \ka Y \cdot \nabla y + y W_y ) f_\Ga + O ( y^\ka ) \text{.}
\end{equation}
In addition, since $Y$ and $y W_y$ are $C^2$ at $\Ga$ (by Definition \ref{lower_order_def}), we have
\begin{align}
\label{duality_2} &( -\partial_t + B_\sigma ) \big[ - \tfrac{1}{ 2 \ka } y^{1+\ka} ( \ka Y \cdot \nabla y + y W_y ) f_\Ga \big] \\
\notag \quad &= - \tfrac{1}{ 2 \ka } [ y^{-\ka} \nabla \cdot ( y^{2\ka} \nabla y ) ] \, ( \ka Y \cdot \nabla y + y W_y ) f_\Ga + O ( y^\ka ) \\
\notag \quad &= - y^{\ka-1} ( \ka Y \cdot \nabla y + y W_y ) f_\Ga + O ( y^\ka ) \text{.}
\end{align}
Summing \eqref{duality_1} and \eqref{duality_2}, and recalling that $\ka \in ( -\frac{1}{2}, 0 )$, we conclude that
\begin{equation}
\label{vf_L2} ( -\partial_t + B_\sigma ) v_f = O ( y^\ka ) \in L^2 ( ( 0, T ) \times \Om ) \text{.}
\end{equation}

Next, we define $v_h$ as the strict solution to the following problem:
\begin{align}
\label{vh} ( -\partial_t + B_\sigma ) v_h = - ( -\partial_t + B_\sigma ) v_f &\quad \text{on $( 0, T ) \times \Omega$,} \\
\notag v_h ( T ) = v_0 &\quad \text{on $\Omega$,} \\
\notag v_h = 0 &\quad \text{on $( 0, T ) \times \Gamma$.}
\end{align}
Note that the existence of $v_h$ follows from Proposition \ref{wp_strict} (adapted to the forward heat equation---see Remark \ref{time_reverse}) along with \eqref{vf_L2}.
Finally, observe that 
\begin{equation}
\label{v_dir} v := v_h + v_f \text{,}
\end{equation}
which lies in the required space, suffices as our desired solution to Problem (C).

Lastly, given $u_T$, $F$, and $u$ as in the hypotheses, we write
\begin{align}
\label{duality_10} \int_{ ( 0, T ) \times \Omega } F v &= \int_{ ( 0, T ) \times \Omega } ( \partial_t u + y^{-\ka} \nabla \cdot [ y^{2\ka} \nabla ( y^{-\ka} u ) ] + X \cdot \nabla u + V_y u ) v \\
\notag &= I_1 + I_2 + I_3 + I_4 \text{,}
\end{align}
and we integrate each term on the right-hand side of \eqref{duality_10} by parts.
First,
\begin{align}
\label{duality_11} I_1 &= \int_{ ( 0, T ) \times \Omega } u ( -\partial_t v ) + \int_\Omega u (T) v (T) - \int_\Omega u (0) v (0) \\
\notag &= \int_{ ( 0, T ) \times \Omega } u ( -\partial_t v ) + \int_\Omega u_T v (T) - \int_\Omega u (0) v_0 \text{.}
\end{align}
To see that the right-hand side of \eqref{duality_11} is well-defined, we consider, for instance,
\[
\int_\Omega u_T v (T) = \int_\Omega u_T \, v_h (T) + \int_\Omega u_T \, v_f (T) \text{.}
\]
Note that both terms on the right-hand side are finite, since $u_T, v_h, v_f \in L^2 ( \Omega )$ by Proposition \ref{wp_strict}, \eqref{vf}, and the assumption $\ka > -\frac{1}{2}$; the remaining term involving $u (0) \, v_0$ is treated similarly.
Next, observe from \eqref{Vy} and \eqref{coeff_duality} that
\begin{align}
\label{duality_12} I_3 + I_4 &= \int_{ ( 0, T ) \times \Omega } u ( Y \cdot \nabla v + W_y v ) - \int_{ ( 0, T ) \times \Gamma } ( Y \cdot \nabla y ) \, \Big[ \lim_{ y \searrow 0 } y^\ka u \Big] \mc{D}_\sigma v \\
\notag &= \int_{ ( 0, T ) \times \Omega } u ( Y \cdot \nabla v + W_y v ) \text{,}
\end{align}
where the boundary term in \eqref{duality_12} vanishes due to \eqref{hidden_dirichlet}.

For the remaining term $I_2$, we first obtain
\begin{align*}
I_2 &= -\int_{ ( 0, T ) \times \Omega } y^\ka \nabla ( y^{-\ka} u ) \cdot y^\ka \nabla ( y^{-\ka} v ) + \int_{ ( 0, T ) \times \Gamma } \mc{N}_\sigma u \mc{D}_\sigma v \\
&= -\int_{ ( 0, T ) \times \Omega } y^\ka \nabla ( y^{-\ka} u ) \cdot y^\ka \nabla ( y^{-\ka} v ) + \int_{ ( 0, T ) \times \Gamma } \mc{N}_\sigma u \, f \text{,}
\end{align*}
where we note that $\mc{N}_\sigma u$ is well-defined by Proposition \ref{neumann_trace}, and where we also note that $y^\ka \nabla ( y^{-\ka} u )$, $y^\ka \nabla ( y^{-\ka} v_h )$, and $y^\ka \nabla ( y^{-\ka} v_f )$ all lie in $C ( [ 0, T ]; L^2 ( \Omega ) )$, by Corollary \ref{prop_hardy_norm}, Proposition \ref{wp_strict}, and \eqref{vf}.
Integrating by parts again then yields
\begin{equation}
\label{duality_13} I_2 = \int_{ ( 0, T ) \times \Omega } u \, \nabla \cdot [ y^{2\ka} \nabla ( y^{-\ka} v ) ] + \int_{ ( 0, T ) \times \Gamma } \mc{N}_\sigma u \, f - \int_{ ( 0, T ) \times \Gamma } \mc{D}_\sigma u \mc{N}_\sigma v \text{.}
\end{equation}
We claim the last term in the right-hand side of \eqref{duality_13} vanishes.
However, treating this (informally written) term properly requires additional comments:
\begin{itemize}
\item For the Neumann trace ``$\mc{N}_\sigma v$", we first notice from Proposition \ref{neumann_trace} that $\mc{N}_\sigma v_h$ is well-defined, with a finite value on $( 0, T ) \times \Ga$.
Also, from \eqref{vf}, we see directly that $\mc{N}_\sigma v_f$ (or, more accurately, $y^{2\ka} D_y ( y^{-\ka} v_f )$ in the limit $y \rightarrow 0$) blows up like $O ( y^{2\ka} )$ at $( 0, T ) \times \Ga$.

\item By the second part of Proposition \ref{neumann_trace}, the Dirichlet trace $\mc{D}_\sigma u$ exists and vanishes to order $O ( y^{1-\ka} )$ at $( 0, T ) \times \Ga$.
\end{itemize}
Thus, the informally stated product ``$\mc{D}_\sigma u \mc{N}_\sigma v$" vanishes at $( 0, T ) \times \Ga$ like $O ( y^{1+\ka} )$, which is a positive power of $y$ since $\ka > -\frac{1}{2}$.

Combining \eqref{Vy} and \eqref{duality_10}--\eqref{duality_13} then yields
\begin{align*}
\int_{ ( 0, T ) \times \Omega } F v &= \int_{ ( 0, T ) \times \Omega } u ( -\partial_t v + B_\sigma v ) + \int_{ ( 0, T ) \times \Gamma } \mc{N}_\sigma u \, f \\
\notag &\qquad + \int_\Omega u_T v (T) - \int_\Omega u (0) v_0 \text{,}
\end{align*}
and the desired identity \eqref{duality} follows from the equation satisfied by $v$.
\end{proof}

\begin{remark}
In proving Proposition \ref{reg_dir}, the extension $v_f := d_\Ga^\ka f_\Ga$ may have seemed natural at first glance.
However, this $v_f$ runs into issues, since $( - \partial_t + B_\sigma ) v_f$ fails to lie in $L^2 ( ( 0, T ) \times \Om )$.
As a result, one requires the extra correction term in \eqref{vf} to ensure $v_f$ is sufficiently well-behaved near the boundary.
In fact, this correction term also motivates the boundary conditions imposed in Definition \ref{lower_order_def}.
\end{remark}

\subsection{Weak Solutions}

The next task is to derive, using the identity \eqref{duality}, a well-posedness theory for Problem (C) that is dual to that of Problem (OI).

\begin{definition} \label{transposition}
Given $v_0 \in H^{-1} ( \Omega )$ and $f \in L^2 ( (0, T) \times \Ga )$, we call
\[
v \in C^0 ( [ 0, T ]; H^{-1} ( \Om ) ) \cap L^2 ( ( 0, T ) \times \Omega )
\]
a \emph{weak} (or \emph{transposition}) \emph{solution} of Problem (C) iff for any $F \in L^2 ( ( 0, T ) \times \Om )$,
\begin{equation}
\label{weak_soln} \int_{ ( 0, T ) \times \Omega } F v = -\int_\Omega u (0) \, v_0 + \int_{ ( 0, T ) \times \Gamma } \mc{N}_\sigma u \, f \text{,}
\end{equation}
where $u$ is the strict solution to Problem (OI) with the above $F$, with $u_T \equiv 0$, and with lower-order coefficients $X$ and $V$ given by \eqref{coeff_duality}.
\end{definition}

\begin{proposition} \label{wp_weak}
Given $v_0 \in H^{-1} ( \Omega )$ and $f \in L^2 ( (0, T) \times \Ga )$, there exists a unique weak solution $v$ of Problem (C).
In addition, $v$ satisfies
\begin{equation}
\label{energy_weak} \| v \|_{ L^\infty ( [ 0, T ]; H^{-1} ( \Omega ) ) }^2 + \| v \|_{ L^2 ( ( 0, T ) \times \Omega ) }^2 \lesssim \| v_0 \|_{ H^{-1} ( \Omega ) }^2 + \| f \|_{ L^2 ( ( 0, T ) \times \Gamma ) }^2 \text{,}
\end{equation}
where the constant depends on $\Om$, $\sigma$, $Y$, $W$.
\end{proposition}

\begin{proof}[Proof sketch.]
Define the linear functional $S: L^2 ( ( 0, T ) \times \Omega ) \rightarrow \R$ by
\[
S F := -\int_\Omega u (0) \, v_0 + \int_{ [ 0, T ] \times \Gamma } \mc{N}_\sigma u \, f \text{,}
\]
where $u$ is the strict solution to Problem (OI) with the above $F$, with $u_T \equiv 0$, and with $X$ and $V$ given by \eqref{coeff_duality}.
Observe that $S$ is bounded, since
\begin{align}
\label{weak_bound_0} | S F |^2 &\lesssim \| u (0) \|_{ H^1 ( \Omega ) }^2 \| v_0 \|_{ H^{-1} ( \Omega ) }^2 + \| \mc{N}_\sigma u \|_{ L^2 ( ( 0, T ) \times \Gamma ) }^2 \| f \|_{ L^2 ( ( 0, T ) \times \Gamma ) }^2 \\
\notag &\lesssim ( \| v_0 \|_{ H^{-1} ( \Omega ) }^2 + \| f \|_{ L^2 ( ( 0, T ) \times \Gamma ) }^2 ) \| F \|_{ L^2 ( ( 0, T ) \times \Omega ) }^2 \text{,}
\end{align}
where in the last step, we applied \eqref{energy_strict} and \eqref{hidden_reg}.
By the Riesz representation theorem, there exists a unique $v \in L^2 ( ( 0, T ) \times \Omega )$ such that
\[
\int_{ ( 0, T ) \times \Omega } F v = S F \text{.}
\]
In particular, $v$ satisfies the desired identity \eqref{weak_soln}.

In addition, the representation theorem and \eqref{weak_bound_0} also imply the estimate
\[
\| v \|_{ L^2 ( ( 0, T ) \times \Omega ) }^2 \lesssim \| v_0 \|_{ H^{-1} ( \Omega ) }^2 + \| f \|_{ L^2 ( ( 0, T ) \times \Gamma ) }^2 \text{,}
\]
hence it remains only to obtain the $C^0 ( [ 0, T ]; H^{-1} ( \Omega ) )$-estimate for $v$.
For this, we fix any $\tau \in ( 0, T ]$ and $u_\tau \in H^1_0 ( \Omega )$, and we let $u$ be the strict solution of
\begin{align}
\label{heat_obs_tau} ( \partial_t + \Delta_\sigma u + X \cdot \nabla + V ) u = 0 &\quad \text{on $( 0, \tau ) \times \Omega$,} \\
\notag u ( \tau ) = u_\tau &\quad \text{on $\Omega$,} \\
\notag u = 0 &\quad \text{on $( 0, \tau ) \times \Gamma$.}
\end{align}
For sufficiently regular $v_0$ and $f$, uniqueness yields that $v$ must be equal to that of Proposition \ref{reg_dir}.
As a result, the identity \eqref{duality} yields
\[
\int_\Omega u_\tau \, v ( \tau ) = \int_\Omega u (0) \, v_0 - \int_{ ( 0, T ) \times \Gamma } \mc{N}_\sigma u \, f \text{,}
\]
The estimate \eqref{energy_strict} then implies
\begin{align*}
\left| \int_\Omega u_\tau \, v ( \tau ) \right| &\leq \| u (0) \|_{ H^1 ( \Omega ) } \| v_0 \|_{ H^{-1} ( \Omega ) } + \| \mc{N}_\sigma u \|_{ L^2 ( ( 0, T ) \times \Gamma ) } \| f \|_{ L^2 ( ( 0, T ) \times \Gamma ) } \\
&\lesssim \| u_\tau \|_{ H^1 ( \Omega ) } [ \| v_0 \|_{ H^{-1} ( \Omega ) } + \| f \|_{ L^2 ( ( 0, T ) \times \Gamma ) } ] \text{,}
\end{align*}
and the desired $C^0 ( [ 0, T ]; H^{-1} ( \Omega ) )$-estimate for $v$ follows.
Finally, the general case $v_0 \in H^{-1} ( \Omega )$ follows via an approximation argument.
\end{proof}

\subsection{Null Controllability}

We can now turn our attention to the main null control result.
The first step is to properly characterize the desired null control:

\begin{proposition} \label{ctl_condition}
Fix any $v_0 \in H^{-1} ( \Omega )$.
Then, $f \in L^2 ( ( 0, T ) \times \Gamma )$ is a null control for Problem (C) (that is, the weak solution $v$ to Problem (C), with the above $v_0$ and $f$, satisfies $v (T) = 0$) if and only if for any $u_T \in H^1_0 ( \Omega )$,
\[
0 = \int_{ ( 0, T ) \times \Gamma } \mc{N}_\sigma u \, f - \int_\Omega u (0) \, v_0 \text{,}
\]
where $u$ is the strict solution of Problem (O), with $( X, V )$ as in \eqref{coeff_duality}.
\end{proposition}

\begin{proof}
For sufficiently regular $v_0$ and $f$, this follows from the identity \eqref{duality} (with $F \equiv 0$).
The general case then follows by approximation.
\end{proof}

\begin{theorem}\label{T.control}
Problem (C) is boundary null controllable---more specifically, given any $v_0 \in H^{-1} ( \Omega )$, there is a null control $f \in L^2 ( ( 0, T ) \times \Gamma )$ for Problem (C).
\end{theorem}

\begin{proof}
Consider the following seminorm on $H^1_0 ( \Omega )$,
\begin{equation}
\label{norm_N} \| u_T \|_{ \mf{N} } := \| \mc{N}_\sigma u \|_{ L^2 ( ( 0, T ) \times \Gamma ) } \text{,} \qquad u_T \in H^1_0 ( \Omega ) \text{,}
\end{equation}
where $u$ is the strict solution of Problem (O), with $u_T$ as above and with $( X, V )$ as in \eqref{coeff_duality}.
Theorem \ref{obs} implies that \eqref{norm_N} defines a norm, and we can now define $\mf{N}$ to be the Hilbert space completion of $H^1_0 ( \Omega )$ with respect to \eqref{norm_N}.

Consider now the functional $I_\sigma: H^1_0 ( \Omega ) \rightarrow \R$ given by
\begin{equation}
\label{I_sigma} I_\sigma ( u_T ) := \tfrac{1}{2} \int_{ ( 0, T ) \times \Gamma } | \mc{N}_\sigma u |^2 - \int_\Omega u (0) \, v_0 \text{,}
\end{equation}
with $u$ as before.
The observability inequality \eqref{obs_ineq} then implies $I_\sigma$ extends to a continuous functional on $\mf{N}$, and this continuity also implies the estimate
\[
I_\sigma ( u_T ) \geq c \| u_T \|_{ \mf{N} }^2 - C \| v_0 \|_{ H^{-1} ( \Omega ) }^2 \text{,} \qquad u_T \in \mf{N} \text{,}
\]
with $c, C > 0$.
In particular, $I_\sigma$ is coercive, hence $I_\sigma$ has a minimizer $u_T^\ast \in \mf{N}$.

Let $\{ u_{ T, j } \}$ be a sequence in $H^1_0 ( \Omega )$ with $u_{ T, j } \rightarrow u_T^\ast$ in $\mf{N}$, and let $\{ u_j \}$ be the corresponding solutions to Problem (O).
By \eqref{obs_ineq} and \eqref{norm_N}, there exist functions $f \in L^2 ( ( 0, T ) \times \Gamma )$ and $u_0 \in H^1_0 ( \Omega )$ such that
\[
\| f - \mc{N}_\sigma u_j \|_{ L^2 ( ( 0, T ) \times \Gamma ) } \rightarrow 0 \text{,} \qquad \| u_0 - u_j (0) \|_{ H^1_0 ( \Omega ) } \rightarrow 0 \text{.}
\]
Finally, taking the first variation of $I_\sigma$ and recalling the above limits, we therefore obtain, for any $u_T \in H^1_0 ( \Omega )$ (and with $u$ as before),
\begin{align*}
0 &= \lim_{ h \rightarrow 0 } \tfrac{1}{h} [ I_\sigma ( u_T^\ast + h u_T ) - I_\sigma ( u_T^\ast ) ] \\
&= \int_{ ( 0, T ) \times \Gamma } \mc{N}_\sigma u \, f - \int_\Omega u (0) \, v_0 \text{.}
\end{align*}
As a result, by Proposition \ref{ctl_condition}, the above $f$ is the desired null control.
\end{proof}

\section*{Acknowledgments}

This work has received funding from the European Research Council under the European Union's Horizon 2020 research and innovation programme (ERC Consolidator Grant agreement 862342, A.E.). A.E.\ is supported in part by the
ICMAT--Severo Ochoa grant CEX2019-000904-S and by RED2018-102650-T funded by MCIN/AEI/10.13039/501100011033.
A.S.\ is supported by the EPSRC grant EP/R011982/1.

\raggedbottom

\end{document}